\documentclass[onefignum,onetabnum]{siamart171218}

\usepackage{amsmath,amssymb,bbm}
\usepackage{color}
\usepackage{mathtools}
\usepackage[T1]{fontenc}
\usepackage{lmodern}
\usepackage[utf8]{inputenc}
\usepackage{microtype}

\usepackage{amsfonts}
\usepackage{graphicx}
\usepackage{multirow}
\usepackage{multicol}
\usepackage{algorithm, algpseudocode}
\usepackage{siunitx}
\usepackage{enumitem}

\Crefname{ALC@unique}{Line}{Lines}

\ifpdf
  \DeclareGraphicsExtensions{.pdf,.eps,.png,.jpg}
\else
  \DeclareGraphicsExtensions{.eps}
\fi

\newsiamremark{remark}{Remark}
\newsiamremark{hypothesis}{Hypothesis}
\crefname{hypothesis}{Hypothesis}{Hypotheses}
\newsiamthm{claim}{Claim}

\newtheorem{assumption}{Assumption}

\headers{Acceleration of Primal-Dual Methods}{Y.\ Liu, Y.\ Xu, and W.\ Yin}

\title{Acceleration of Primal-Dual Methods by Preconditioning and Simple Subproblem Procedures\thanks{Submitted to the editors XXX 2018.
\funding{The work of Y. Liu and W. Yin is supported in part by NSF DMS-1720237 and ONR N000141712162.}}}

\author{Yanli Liu\thanks{Department of Mathematics, University of California, Los Angeles, CA
  (\email{yanli@math.ucla.edu}, \email{wotaoyin@math.ucla.edu}).}
\and Yunbei Xu\thanks{Graduate School of Business, Columbia University, New York, NY
  (\email{yunbei.xu@gsb.columbia.edu}).}
\and Wotao Yin\footnotemark[2]}

\usepackage{amsopn}

\ifpdf
\hypersetup{
  pdftitle={Acceleration of Primal-Dual Methods},
  pdfauthor={Yanli Liu, Yunbei Xu, and Wotao Yin}
}
\fi

\DeclareMathOperator\prox{Prox}
\DeclareMathOperator*{\argmin}{arg\,min}

\DeclareMathOperator*{\R}{\mathbb{R}}
\DeclareMathOperator*{\RR}{\mathbb{R}\cup\{+\infty\}}
\DeclareMathOperator*{\Rn}{\mathbb{R}^n}
\DeclareMathOperator*{\Rm}{\mathbb{R}^m}

\newcommand{\dom}{\mathbf{dom}}

\newcommand{\vz}{\mathbf{0}}
\makeatletter

\renewcommand{\paragraph}{%
  \@startsection{paragraph}{4} {\z@}
  {1ex \@plus 1ex \@minus .2ex}
  {-1ex}%
  {\bfseries}%
}

\makeatother

\begin{document}

\maketitle

\begin{abstract}
Primal-Dual Hybrid Gradient (PDHG) and Alternating Direction Method of Multipliers (ADMM) are two widely-used first-order optimization methods. They reduce a difficult problem to simple subproblems, so they are easy to implement and have many applications. As first-order methods, however, they are sensitive to problem conditions and can struggle to reach the desired accuracy. To improve their performance, researchers have proposed techniques such as diagonal preconditioning and inexact subproblems. This paper realizes additional speedup about one order of magnitude.

Specifically, we choose non-diagonal preconditioners that are much more effective than diagonal ones. Because of this, we lose closed-form solutions to some subproblems, but we found simple procedures to replace them such as a few proximal-gradient iterations or a few epochs of proximal block-coordinate descent, which are in closed forms. We show global convergence while fixing the number of those steps in every outer iteration. Therefore, our method is reliable and straightforward.

Our method opens the choices of preconditioners and maintains both low per-iteration cost and global convergence. Consequently, on several typical applications of primal-dual first-order methods, we obtain 4--95$\times$ speedup over the existing state-of-the-art.
\end{abstract}


\section{Introduction}

In this paper, we consider the following optimization problem:
\begin{align}
    \mathop{\mathrm{minimize}}_{x\in \mathbb{R}^n} f(x)+g(Ax),
    \label{equ: PDHG problem}
\end{align}
together with its dual problem:
\begin{align}
    \mathop{\mathrm{minimize}}_{z\in \mathbb{R}^m} f^*(-A^Tz)+g^*(z),
    \label{equ: PDHG dual problem}
\end{align}
where $f:\Rn\rightarrow \R\cup\{+\infty\}$ and $g: \Rm\rightarrow \R\cup\{+\infty\}$ are closed proper convex, and $A\in \mathbb{R}^{m\times n}$ is a matrix, $f^*$ and $g^*$ are the convex conjugates of $f$ and $g$, respectively.

Formulations \eqref{equ: PDHG problem} or \eqref{equ: PDHG dual problem} are abstractions of many application problems, which include image restoration \cite{zhu2008efficient}, magnetic resonance imaging \cite{valkonen2014primal}, network optimization \cite{feijer2010stability}, computer vision \cite{pock2009algorithm}, and earth mover's distance \cite{li2017parallel}.

To solve \eqref{equ: PDHG problem}, one can apply primal-Dual algorithms such as Primal-Dual Hybrid Gradient (PDHG) and Alternating Direction Method of Multipliers (ADMM).
However, as first-order algorithms, PDHG and ADMM suffer from slow (tail) convergence especially on poorly conditioned problems, when they may take thousands of iterations and still struggle reaching just four digits of accuracy. While they have many other advantages such as being easy to implement and friendly to parallelization, having their performance very sensitive to problem conditions is their main disadvantage. To improve the performance of PDHG and ADMM, researchers have tried using preconditioners, but for reasons we discuss below, only diagonal preconditioners so far.
Depending on the application and how one applies splitting, PDHG and ADMM may or may not have subproblems with closed-form solutions. When they do not, researchers have studied approximate subproblem solutions to reduce the total running time.
In the next subsection, we review the relevant works of preconditioning and inexact subproblems. 

\subsection{Background}
Many problems to which we apply PDHG have separable functions $f$ or $g$, or both, so the resulting PDHG subproblems often (though not always) have closed-form solutions. When subproblems are simple, we care mainly about the convergence rate of PDHG, which depends on the problem conditioning. 
To accelerate PDHG, diagonal preconditioning \cite{pock2011diagonal}
was proposed since its diagonal structure maintains closed-form solutions for the subproblems and, therefore, reduces iteration complexity without making each iteration more difficult. In comparison, non-diagonal preconditioners are much more effective at reducing iteration complexity, but their off-diagonal entries couple different components in the subproblems, causing the lost of closed-form solutions of subproblems. So, it may appear we cannot have both fewer iterations and simple subproblems at the same time. 

When a PDHG subproblem has no closed-form solution, one often uses an iterative algorithm to approximately solve it. We call it Inexact PDHG. Under certain conditions, Inexact PDHG still converges to the exact solution. Specifically, \cite{rasch2018inexact} uses three different types of conditions to skillfully control the errors of the subproblems; all those errors need to be summable over all the iterations and thereby requiring the error to diminish asymptotically. 
In an interesting method from \cite{bredies2015preconditioned, bredies2017proximal}, one subproblem computes a proximal operator of a convex quadratic function, which can include a preconditioner and still has a closed-form solution involving matrix inversion. This proximal operator is successively applied $n$ times in each iteration, for $n\geq 1$. 

ADMM has different subproblems. One of it subproblems minimizes the sum of $f(x)$ and a squared term involving $Ax$. Only when $A$ has special structures does the subproblem have closed-form solutions. Inexact ADMM refers to the ADMM with at least one of its subproblems inexactly solved. An \textit{absolute error criterion} was introduced in \cite{eckstein1992douglas}, where the subproblem errors are controlled by a summable (thus diminishing) sequence of error tolerances.
To simplify the choice of the sequences, a \textit{relative error criterion} was adopted in several later works, where the subproblem errors are controlled by a single parameter multiplying certain quantities that one can compute during the iterations. In \cite{ng2011inexact}, the parameters need to be square summable.
In \cite{li2013inexact}, the parameters are constants when both objective functions are Lipschitz differentiable. In \cite{eckstein2017relative, eckstein2017approximate}, two possible outcomes of the algorithm are described: (i) infinite outer loops and finite inner loops, and (ii) finite outer loops and the last inner loop is infinite, both guaranteeing convergence to a solution. On the other hand, it is unclear how to compare them. Since there is no bound on the number of inner loops in case (i), one may recognize it as case (ii) and stop the algorithm before it converges.

There are works that apply certain kinds of preconditioning to accelerate ADMM. Paper \cite{giselsson2014diagonal} uses diagonal preconditioning and observes improved performance. After that, non-diagonal preconditioning is analyzed \cite{bredies2015preconditioned, bredies2017proximal}, which presents effective preconditioners for specific applications. One of their preconditioners needs to be inverted (though not needed in our method). Recently, preconditioning for strongly convex problems has also been studied \cite{giselsson2017linear} and promising numerical performances.

\subsection{Contributions}


Simply speaking, we find a way to have both non-diagonal preconditioners (thus much fewer iterations) and very simple subproblem procedures.

Our exposition takes a few steps. First, we present Preconditioned PDHG (PrePDHG) and discuss how to choose preconditioners by minimizing an upper bound in PrePDHG's ergodic convergence analysis. 
We can observe ADMM as a special case of PrePDHG where one of the preconditioners is identity (no preconditioning) and the other is the optimal choice, which minimizes the bound, and, thereby, explaining why ADMM often takes fewer iterations than PDHG.

Then, we show that PrePDHG still converges when one of its subproblems is solved inexactly to a specified condition. Remarkably, we do not need to verify this condition to stop a procedure since it is automatically satisfied as long as one applies a common iterative method for a \emph{fixed number} of iterations. Common choices of subproblem procedures include proximal gradient descent, FISTA with restart, proximal block coordinate descent,
and accelerated block-coordinate-gradient-descent (BCGD) methods (e.g.,  \cite{lin2015accelerated,allen2016even, hannah2018texttt}). We call this method iPrePDHG (i for ``inexact''). 

We leave the other subproblem exactly solved in iPrePDHG since we have not encountered interesting applications that require non-diagonal preconditioners for both subproblems yet. If one is encountered, we can always split it in a way such that all ill-conditioned terms are collected in one subproblem.

Next, we apply iPrePDHG and develop effective preconditioners for a set of classic and representative applications of primal-dual splitting methods: image denoising, graph cut, optimal transport, and CT reconstruction. 
The CT reconstruction application uses a diagonal preconditioner in one subproblem, which has a closed-form solution, and a non-diagonal preconditioner in the other, which has no closed-form solution. In each of the other applications, one subproblem uses no preconditioner, and the other uses a non-diagonal preconditioner.

Finally, we numerically evaluated the performance of iPrePDHG 
using our recommended preconditioners. 
We obtained speedups of 4--95 times over the existing state-of-the-art.
We believe it is a sufficient demonstration on how to apply preconditioners effectively and efficiently in PDHG.

Since we show ADMM is a special PrePDHG, our method also applies to ADMM. In fact, the iPrePDHG algorithms for three of the four applications are also Inexact Preconditioned ADMM under simple transformations. 


\subsection{Organization}
The rest of this paper is organized as follows: Section \ref{Section: preliminaries} establishes notation and reviews basics. In the first part of Section \ref{Section: Acceleration of PrePDHG}, we provide a criterion for choosing preconditioners. In its second part, we introduce the condition for inexact subproblems, which can be automatically satisfied by iterating a fixed number of certain inner loops. This method is iPrePDHG. In the last part of Section \ref{Section: Acceleration of PrePDHG}, we establish the convergence of iPrePDHG. Section \ref{Section: Numerical experiments} describes specific preconditioners and reports numerical results. Finally, Section \ref{Sec: conclusion} concludes the paper.

\section{Preliminaries}
\label{Section: preliminaries}
In this section, we introduce our notation and state the basic results that we need later. 

We use $\|\cdot\|$ for  $\ell_2-$norm and $\langle\cdot, \cdot\rangle$ for  dot product. $M\succ 0$ means $M$ is a  symmetric, positive definite matrix, 
and $M\succeq 0$ means $M$ is a symmetric, positive semidefinite matrix.

We write $\lambda_{\text{min}}(M)$ and $\lambda_{\text{max}}(M)$ as the smallest and the largest eigenvalues of $M$, respectively, and  $\kappa(M)=\frac{\lambda_{\text{max}}(M)}{\lambda_{\text{min}}(M)}$ as the condition number of $M$. For $M\succeq 0$, let $\|\cdot\|_M$ and $\langle \cdot, \cdot \rangle_M$ denote the semi-norm and inner product induced by $M$, respectively. If $M\succ 0$, $\|\cdot\|_M$ is a norm.

For a proper closed convex function $\phi:\Rn\rightarrow\RR$, its subdifferential at $x\in\dom f$ is written as
$$
\partial \phi(x)=\{v\in \Rn\,|\, \phi(z)\geq \phi(x)+\langle v, z-x\rangle\,\,\forall z\in\Rn\},$$
and its convex conjugate as
$$\phi^*(y)=\sup_{x\in \R^n}\{ \langle y, x\rangle-\phi(x) \}. $$
We have $y\in\partial \phi(x)$ if and only if $x\in\partial \phi^*(y)$.

For any  $M\succ 0$, we define the extended proximal operator of $\phi$ as
\begin{align}
\prox^M_{\phi}(x)\coloneqq\argmin_{y\in\Rn}\{\phi(y)+\frac{1}{2}\|y-x\|^2_{M}\}.
\label{equ: extended proximal operator}
\end{align}
If $M=\gamma^{-1} I$ for $\gamma>0$, it reduces to a classic proximal operator.

We also have the following generalization of Moreau's Identity:
\begin{lemma}[\cite{combettes2013moreau}, Theorem 3.1(ii)]
For any proper closed convex function $\phi$ and $M\succ 0$, we have
\begin{equation}
    x=\prox^M_{\phi}(x)+M^{-1}\prox^{M^{-1}}_{\phi^*}(Mx).
    \label{equ: generalized Moreau}
\end{equation}
\end{lemma}

We say a proper closed function is  a Kurdyka-Łojasiewicz (KŁ) function if, for each $x_0\in \textbf{dom} f$, there exist $\eta\in(0,\infty]$, a neighborhood $U$ of $x_0,$ and a continuous concave function $\varphi: [0, \eta)\rightarrow \R_+$ such that:
\begin{enumerate}
    \item $\varphi(0) = 0$,
    \item $\varphi$ is $C^1$ on $(0,\eta)$,
    \item for all $s\in(0,\eta), \varphi'(s)>0$,
    \item for all $x\in U\cap \{x\,|\,f(x_0)<f(x)<f(x_0)+\eta\}$, the KŁ inequality holds:
    \[
    \varphi'(f(x)-f(x_0))\text{dist}(0,\partial f(x))\geq 1.
    \]
\end{enumerate}

\section{Main results}
\label{Section: Acceleration of PrePDHG}
This section presents our results on preconditioner selection, the proposed method iPrePDHG, and its convergence.

Throughout this section, we assume the following regularity assumptions:
\begin{assumption}
\label{Assumption 1}
\hfill
\begin{enumerate}
    \item $f:\Rn\rightarrow \RR$, $g:\Rm\rightarrow \RR$ are proper closed convex.
    \item A primal-dual solution pair $(x^{\star}, z^{\star})$ of \eqref{equ: PDHG problem} and \eqref{equ: PDHG dual problem} exists, i.e.,
    \[
       \mathbf{0}\in\partial f(x^{\star})+A^Tz^{\star},\quad \mathbf{0}\in \partial g(Ax^{\star})-z^{\star}.
    \]
\end{enumerate}
\end{assumption}
The problem \eqref{equ: PDHG problem} also has the following convex-concave saddle-point formulation:
\begin{align}
\min_{x\in \mathbb{R}^n}\max_{z\in \mathbb{R}^m} \varphi(x,z)\coloneqq f(x)+\langle Ax, z\rangle-g^*(z).
\label{equ: convex-concave}
\end{align}
A primal-dual solution pair $(x^{\star}, z^{\star})$ is a solution of \eqref{equ: convex-concave}.

\subsection{Preconditioned PDHG}
\label{subsection: PrePDHG}
The method of Primal-Dual Hybrid Gradient or PDHG \cite{zhu2008efficient,chambolle2011first} for solving \eqref{equ: PDHG problem} refers to the iteration
\begin{align}
\begin{split}
    {x}^{k+1}&=\prox_{\tau f}(x^{k}-\tau A^Tz^{k}),\\
    {z}^{k+1}&=\prox_{\sigma g^*}(z^{k}+\sigma A(2{x}^{k+1}-x^{k})).
    \label{equ: PDHG}
\end{split}
\end{align}

When $\frac{1}{\tau\sigma}\geq \|A\|^2$, the iterates of \eqref{equ: PDHG} converge \cite{chambolle2011first} to a primal-dual solution pair of \eqref{equ: PDHG problem}.
We can generalize \eqref{equ: PDHG} by applying preconditioners $M_1, M_2\succ 0$ (their choices are discussed below) to obtain Preconditioned PDHG or PrePDHG:
\begin{align}
\label{equ: PrePDHG}
\begin{split}
    {x}^{k+1}&=\prox^{M_1}_{ f}(x^{k}-M_1^{-1} A^Tz^{k}),\\
    {z}^{k+1}&=\prox^{M_2}_{ g^*}(z^{k}+M_2^{-1} A(2{x}^{k+1}-x^{k})),
\end{split}
\end{align}
where the extended proximal operators $\prox^{M_1}_{f}$ and $\prox^{M_2}_{g^*}$ are defined in \eqref{equ: extended proximal operator}. We can obtain the convergence of PrePDHG using the analysis in \cite{chambolle2016ergodic}. 

There is no need to compute $M_1^{-1}$ and $M_2^{-1}$ since \eqref{equ: PrePDHG} is equivalent to
\begin{align}
\label{equ: PrePDHG 1}
\begin{split}
    {x}^{k+1}&=\argmin_{x\in\Rn}\{f(x)+\langle x-x^k, A^Tz^k\rangle + \frac{1}{2}\|x-x^k\|^2_{M_1}\},\\
    {z}^{k+1}&=\argmin_{z\in\Rm}\{g^*(z)-\langle z-z^k, A(2x^{k+1}-x^k)\rangle+\frac{1}{2}\|z-z^k\|^2_{M_2}\}.
\end{split}
\end{align}

\subsection{Choice of preconditioners}
\label{Subsection: Convergence of PrePDHG and Choice of Preconditioners}
In this section, we discuss how to select appropriate preconditioners $M_1$ and $M_2$. 
As a by-product, we show that ADMM corresponds to choosing $M_1=\frac{1}{\tau} I_{n\times n}$ and optimally choosing $M_2=\tau AA^T$, thereby, explaining why ADMM appears to be faster than PDHG. 

Let us start with the following lemma, which characterizes primal-dual solution pairs of \eqref{equ: PDHG problem} and \eqref{equ: PDHG dual problem}.
\begin{lemma}
\label{lem: choose preconditioners}
Under Assumption \ref{Assumption 1}, $(X, Z)$ is a primal-dual solution pair of \eqref{equ: PDHG problem} if and only if $\varphi(X, z)-\varphi(x,Z)\leq 0$ for any $(x,z)\in\mathbb{R}^{n+m}$, where $\varphi$ is given in the saddle-point formulation \eqref{equ: convex-concave}.
\end{lemma}
\begin{proof}
If $(X,Z)$ is a primal-dual solution pair of \eqref{equ: PDHG problem}, then
\[
-A^TZ\in\partial f(X),\quad AX\in \partial g^*(Z).
\]
Hence, for any $(x,z)\in\mathbb{R}^{n+m}$ we have
\[
f(x)\geq f(X)+\langle -A^TZ, x-X \rangle,\quad g^*(z)\geq g^*(Z)+\langle AX, z-Z\rangle.
\]
Adding them together yields $\varphi(X, z)-\varphi(x,Z)\leq 0$.

On the other hand, if $\varphi(X, z)-\varphi(x,Z)\leq 0$ for any $(x,z)\in\mathbb{R}^{n+m}$, then
\[
\langle AX, z\rangle +f(X)-g^*(z)-\langle Ax, Z\rangle-f(x)+g^*(Z)\leq 0 \quad \text{for any}\,\,(x,z)\in\mathbb{R}^{n+m}.
\]
Taking $x=X$ yields $\langle AX, z-Z\rangle-g^*(z)+g^*(Z)\leq 0$, so $AX\in\partial g^*(Z)$; Similarly, taking $z=Z$ gives $\langle AX-Ax, Z\rangle+f(X)-f(x)\leq 0$, so $-A^TZ\in\partial f(X)$. As a result, $(X,Z)$ is a primal-dual solution pair of \eqref{equ: PDHG problem}.
\end{proof}

We present the following convergence result,  adapted from Theorem 1
of \cite{chambolle2016ergodic}.
\begin{theorem}
\label{thm: ergodic convergence}
Let $(x^k, z^k), k=0,1,...,N$ be a sequence generated by PrePDHG \eqref{equ: PrePDHG}. Under Assumption \ref{Assumption 1}, if in addition
\begin{equation}
\label{equ: M}
    \tilde{M}\coloneqq\begin{pmatrix}
    M_1 & -A^T\\
    -A & M_2
    \end{pmatrix}\succeq 0,
\end{equation}
then, for any $x\in\Rn$ and $z\in\Rm$, it holds that
\begin{equation}
\label{equ: ergodic convergence}
    \varphi(X^N, z)-\varphi(x, Z^N)\leq \frac{1}{2N}(x-x^0, z-z^0)\begin{pmatrix}
    M_1 & -A^T\\
    -A & M_2
    \end{pmatrix}
    \begin{pmatrix}
    x-x^0\\
    z-z^0
    \end{pmatrix},
\end{equation}
where $X^N=\frac{1}{N}\sum_{i=1}^{N}x^i$ and $ Z^N=\frac{1}{N}\sum_{i=1}^{N}z^i$.
\end{theorem}
\begin{proof}
This follows from Theorem 1 of \cite{chambolle2016ergodic} by setting $L_f=0$, $\frac{1}{\tau}D_x(x,x_0)=\frac{1}{2}\|x-x^0\|^2_{M_1}$, $\frac{1}{\sigma}D_z(z,z_0)=\frac{1}{2}\|z-z^0\|^2_{M_2}$, and $K=A$. Note that in Remark 1 of \cite{chambolle2016ergodic}, $D_x$ and $D_z$ need to be $1-$strongly convex to ensure that equation (13) holds, which is exactly our \eqref{equ: M}. Therefore, we don't need $D_x$ and $D_z$ to be strongly convex.
\end{proof}

Based on the above results,
one approach to accelerate convergence is to choose preconditioners $M_1$ and $M_2$ to obey \eqref{equ: M} and minimize the right-hand side of \eqref{equ: ergodic convergence}. When a pair of preconditioner matrices attains this minimum, we say they are optimal. When one of them is fixed, the other that attains the minimum is also called optimal.  

By Schur complement, the condition \eqref{equ: M} is equivalent to $M_2\succeq AM_1^{-1}A^T$. Hence, for any given $M_1\succ 0$, the optimal $M_2$ is $AM_1^{-1}A^T$.

Original PDHG \eqref{equ: PDHG} corresponds to $M_1=\frac{1}{\tau}I_{n\times n}$, $M_2=\frac{1}{\sigma}I_{m\times m}$ with $\tau$ and $\sigma$ obeying $\frac{1}{\tau\sigma}\geq \|A\|^2$ for convergence. In  Appendix \ref{App: to ADMM}, we show that ADMM for problem \eqref{equ: PDHG problem} corresponds to setting $M_1=\frac{1}{\tau} I_{n\times n}, ~M_2={\tau} AA^T$, $M_2$ is optimal since $AM_1^{-1}A^T=\tau AA^T=M_2$. (This is related to, but different from, the result in \cite[Sec. 4.3]{chambolle2011first} stating that PDHG is equivalent to a preconditioned ADMM.)


By using more general pairs of $M_1,M_2$, we can potentially have even fewer iterations of PrePDHG than ADMM. 

\subsection{PrePDHG with fixed inner iterations}
\label{subsection: finite inner loops}
It wastes total time to solve the subproblems in \eqref{equ: PrePDHG 1} very accurately. It is more efficient to develop a proper condition and stop the subproblem procedure, which we call \emph{inner iterations}, 
once the condition is satisfied. 
It is even better if we can simply fix the number of inner iterations and still guarantee global convergence. 

In this subsection, we describe the ``bounded relative error'' of the $z$-subproblem in \eqref{equ: PrePDHG} and then show that this can be satisfied by running a fixed number of inner iterations, uniformly for every outer loop.
\begin{definition}[Bounded relative error condition]
\label{def: bounded relative error}
Given $x^k$, $x^{k+1}$ and $z^k$, we say that the $z$-subproblem in PrePDHG~\eqref{equ: PrePDHG} is solved to a bounded relative error if there is a constant $c>0$ such that
\begin{align}
    &\mathbf{0}\in\partial g^*(z^{k+1})+M_2\big(z^{k+1}-z^{k}-M_2^{-1}A(2x^{k+1}-x^k)\big)+\varepsilon^{k+1},\label{equ: varepsilon^{k+1}}\\
    &\|\varepsilon^{k+1}\| \leq c\|z^{k+1}-z^{k}\|.\label{equ: bounded}
\end{align}
\end{definition}

Remarkably, this condition does not need to be checked at run time. For a fixed $c>0$, the condition can be satisfied by a fixed number of inner iterations using, for example,  proximal gradient iteration (Theorem \ref{thm: finite inner loops}). One can also use faster solvers, e.g., FISTA with restart \cite{o2015adaptive}, and solvers that suit the subproblem structure, e.g., cyclic proximal BCD (Theorem \ref{thm: BCD finite inner loops}). Although the error in solving $z$-subproblems appears to be neither summable nor square summable, convergence can still be established. But first, we summarize this method in Algorithm \ref{alg: Inexact PrePDHG}.

\begin{algorithm}
\caption{Inexact Preconditioned PDHG or iPrePDHG}
\label{alg: Inexact PrePDHG}
\textbf{Input:} $f,g,A$ in \eqref{equ: PDHG problem}, preconditioners $M_1$ and $M_2$, initial  $(x_0, z_0)$, $z$-subproblem iterator $S$, inner iteration number $p$, max outer iteration number $K$.\\
\textbf{Output:} $(x^K, z^K)$
\begin{algorithmic}[1]
\For{$k\leftarrow 0, 1,..., K-1$}{}
\State{${x}^{k+1}=\prox^{M_1}_{ f}(x^{k}-M_1^{-1} A^Tz^{k})$;}
\State{$z^{k+1}_0=z^k$;}
\For{$i\leftarrow 0, 1,...,p-1$}
\State{$z^{k+1}_{i+1}=S(z^{k+1}_i, x^{k+1}, x^k)$;}
\EndFor
\State{$z^{k+1}=z^{k+1}_p$;}
\Comment{which approximates $\prox^{M_2}_{g^*}(z^{k}+M_2^{-1} A(2{x}^{k+1}-x^{k}))$}
\EndFor
\end{algorithmic}
\end{algorithm}

\begin{theorem}
\label{thm: finite inner loops}
Take Assumption \ref{Assumption 1}. Suppose in iPrePDHG, or Algorithm \ref{alg: Inexact PrePDHG}, we choose $S$ as the proximal-gradient step with stepsize $\gamma\in (0, \frac{2\lambda_{\mathrm{min}}(M_2)}{\lambda_{\mathrm{max}}^2(M_2)})$ and repeat it $p$ times, where $p \geq 1$. Then, $z^{k+1}=z^{k+1}_p$ is an approximate solution to the $z$-subproblem up to a bounded relative error in \eqref{equ: bounded} for
\begin{align}
\label{equ: c}
    c=c(p)=\frac{\frac{1}{\gamma}+\lambda_{\mathrm{max}}(M_2)}{1-\tau^p}(\tau^p+\tau^{p-1}),
\end{align}
where $\tau=\sqrt{1-\gamma(2\lambda_{\mathrm{min}}(M_2)-\gamma \lambda_{\mathrm{max}}^2(M_2))}<1$.
\end{theorem}
\begin{proof}
The $z$-subproblem in \eqref{equ: PrePDHG 1} is of the form
\begin{align}
\label{equ: z subproblem}
\mathop{\mathrm{minimize}}_{z\in\Rm} h_1(z)+h_2(z),
\end{align}
for
$    h_1(z)=g^*(z)$ and 
$    h_2(z)=\frac{1}{2}\|z-z^k-M_2^{-1}A(2x^{k+1}-x^k)\|^2_{M_2}.$
With our choice of $S$ as the proximal-gradient descent step, the inner iterations are
\begin{align}
    z^{k+1}_0&=z^{k},\nonumber\\
    z^{k+1}_{i+1}&=\prox_{\gamma h_1}(z^{k+1}_{i}-\gamma \nabla h_2(z^{k+1}_{i})), \quad i=0, 1,...,p-1,\label{eq:proxgrad}
\end{align}
Concerning the last iterate $z^{k+1}=z^{k+1}_p$, we have from the definition of $\prox_{\gamma h_1}$ that
\[
\vz\in \partial h_1(z^{k+1}_p)+\nabla h_2(z^{k+1}_{p-1})+\frac{1}{\gamma}(z^{k+1}_{p}-z^{k+1}_{p-1}).
\]
Compare this with \eqref{equ: varepsilon^{k+1}} and use $z^{k+1}=z^{k+1}_p$ to get
\[
\varepsilon^{k+1}=\frac{1}{\gamma}(z^{k+1}_p-z^{k+1}_{p-1})+\nabla h_2(z^{k+1}_{p-1})-\nabla h_2(z^{k+1}_{p}).
\]
It remains to show that $\varepsilon^{k+1}$ satisfies \eqref{equ: bounded}.

Let $z^{k+1}_{\star}$ be the solution of \eqref{equ: z subproblem}, $\alpha=\lambda_{\text{min}}(M_2)$, and $\beta=\lambda_{\text{max}}(M_2)$. Then $h_1(z)$ is convex and $h_2(z)$ is $\alpha$-strongly convex and $\beta$-Lipschitz differentiable. Consequently, \cite[Prop. 26.16(ii)]{bauschke2017convex} gives
\[
\|z^{k+1}_{i}-z^{k+1}_{\star}\|\leq \tau^{i} \|z^{k+1}_0-z^{k+1}_{\star}\|, \quad \forall i= 0, 1, ..., p,
\]
where $\tau=\sqrt{1-\gamma(2\alpha-\gamma \beta^2)}$.

Let $a_i = \|z^{k+1}_i-z^{k+1}_{\star}\|$. Then, $a_i\leq \tau^i a_0$. We can derive
\begin{align}
\label{equ: 1}
\|\varepsilon^{k+1}\|&\leq (\frac{1}{\gamma}+\beta)\|z^{k+1}_p-z^{k+1}_{p-1}\|\leq (\frac{1}{\gamma}+\beta)(a_p+a_{p-1})\leq (\frac{1}{\gamma}+\beta)(\tau^p+\tau^{p-1})a_0.
\end{align}
On the other hand, we have
\begin{align}
\|z^{k+1}-z^k\|&\geq a_0-a_p\geq (1-\tau^p)a_0.\label{equ: 2}
\end{align}
Combining these two equations yields
\[
\|\varepsilon^{k+1}\|\leq c\|z^{k+1}-z^k\|,
\]
where $c$ is given in \eqref{equ: c}.
\end{proof}
Theorem \ref{thm: finite inner loops} uses the iterator $S$ that is the proximal-gradient step. It is straightforward to extend its proof to $S$ being the FISTA step. We omit the proof.

In our next theorem, we let $S$ be the iterator of one epoch of the cyclic proximal BCD method. A BCD method updates one block of coordinates at a time while fixing the remaining blocks. In one epoch of cyclic BCD, all the blocks of coordinates are sequentially updated, and every block is updated once. In cyclic \emph{proximal} BCD, each block of coordinates is updated by a proximal-gradient step, just like \eqref{eq:proxgrad} except only the chosen block is updated each time. When $h_1$ is block separable, each update costs only a fraction of updating all the blocks together. When different blocks are updated one after another, the Gauss-Seidel effect brings more progress. In addition, since the Lipschitz constant of each block gradient of $h_2$ is typically less than than that of $\nabla h_2$, one can use a larger stepsize $\gamma$ and get potentially even faster progress. Therefore, the iterator of cyclic proximal BCD is a better choice for $S$.

\begin{theorem}
\label{thm: BCD finite inner loops}
Let Assumption \ref{Assumption 1} hold and $g$ be block separable, i.e., \\$z=(z_1,z_2,...,z_l)$ and $g(z)=
\sum_{i=1}^l g_i(z_i)$. Suppose in iPrePDHG, or Algorithm \ref{alg: Inexact PrePDHG}, we choose $S$ as the iterator of cyclic proximal BCD with stepsize $\gamma$ satisfying
\begin{align*}
0<\gamma\leq \min\bigg\{&\frac{2\lambda_{\mathrm{min}}(M_2))}{\lambda_{\mathrm{max}}^2(M_2))},\frac{1-\sqrt{1-\gamma(2\lambda_{\mathrm{min}}(M_2)-\gamma \lambda_{\mathrm{max}}^2(M_2))}}{4\sqrt{2}\gamma l\lambda_{\mathrm{max}}(M_2)},\\ &\frac{1}{4l\lambda_{\mathrm{max}}(M_2)}, \frac{2l\lambda_{\mathrm{max}}(M_2)}{17l\lambda_{\mathrm{max}}(M_2)+2(\frac{1-\sqrt{1-\gamma(2\lambda_{\mathrm{min}}(M_2)-\gamma \lambda_{\mathrm{max}}^2(M_2))}}{\gamma})^2}\bigg\},
\end{align*}
and we set $p \geq 1$.
Then,
$z^{k+1}=z_p^{k+1}$ is an approximate solution to the $z$-subproblem up to a bounded relative error \eqref{equ: bounded} for
\begin{align}
\label{equ: BCD c}
    c=c(p)=\frac{(l\lambda_{\mathrm{max}}(M_2)+\frac{1}{\gamma})(\rho^p+\rho^{p-1})}{1-\rho^p},
\end{align}
where $\rho=1-\frac{\big(1-\sqrt{1-\gamma(2\lambda_{\mathrm{min}}(M_2)-\gamma \lambda_{\mathrm{max}}^2(M_2))}\big)^2}{2\gamma}<1$.
\end{theorem}
\begin{proof}
See Appendix \ref{App: BCD proof}.
\end{proof}

\subsection{Global convergence of iPrePDHG}

In this subsection, we proceed to establishing the convergence of Algorithm \ref{alg: Inexact PrePDHG}. Our approach first transforms Algorithm \ref{alg: Inexact PrePDHG} into an equivalent algorithm in Proposition \ref{thm: Inexact PLADMM} below and then proves its convergence in Theorems \ref{thm: subsequence convergence} and \ref{thm: sequence convergence} below. 

First, let us show that PrePDHG \eqref{equ: PrePDHG} is equivalent to an algorithm applied on the dual problem \eqref{equ: PDHG dual problem}. This equivalence is analogous to the equivalence between PDHG \eqref{equ: PDHG} and Linearized ADMM applied to the dual problem \eqref{equ: PDHG dual problem}, shown in  \cite{esser2010general}). 
Specifically, PrePDHG is equivalent to
\begin{align}
\begin{split}
    z^{k+1}&=\prox^{M_2}_{g^*}(z^{k}+M_2^{-1}AM_1^{-1}(-A^Tz^{k}-y^{k}+u^{k})),\\
    y^{k+1}&=\prox^{M_1^{-1}}_{f^*}(u^{k}-A^Tz^{k+1}),\\
    u^{k+1}&=u^{k}-A^Tz^{k+1}-y^{k+1}.
    \label{equ: PLADMM}
\end{split}
\end{align}
When $M_1=\frac{1}{\tau}I, M_2=\lambda I$, \eqref{equ: PLADMM} reduces to Linearized ADMM, also known as Split Inexact Uzawa \cite{zhang2011unified}.

Furthermore, iPrePDHG in Algorithm \ref{alg: Inexact PrePDHG} is equivalent to \eqref{equ: PLADMM} with inexact subproblems, which we present in Algorithm \ref{alg: Inexact PLADMM}.
\begin{algorithm}[H]
\caption{Inexact Preconditioned ADMM}
\label{alg: Inexact PLADMM}
\textbf{Input:} $f^*: \Rn\rightarrow \R, g^*: \Rm\rightarrow\R$, $A\in\mathbb{R}^{m\times n}$, preconditioners $M_1$ and $M_2$,\\initial vector $(z_0, y_0, u_0)$, subproblem solver $S$ for the $z$-subproblem in \eqref{equ: PLADMM}, number of inner loops $p$, number of outer iterations $K$.\\
\textbf{Output:} $(z^K, y^K, u^K)$
\begin{algorithmic}[1]
\For{$k\leftarrow 0, 1,..., K-1$}{}
\State{$z^{k+1}_0=z^k$;}
\For{$i\leftarrow 0, 1,...,p-1$}{}
\State{$z^{k+1}_{i+1}=S(z^{k+1}_i, y^{k}, u^k)$;}
\EndFor
\State{$z^{k+1}=z^{k+1}_p$;}
\Comment{\text{approximate} $\prox^{M_2}_{g^*}(z^{k}+M_2^{-1}AM_1^{-1}(-A^Tz^{k}-y^{k}+u^{k}))$.}
\State{$y^{k+1}=\prox^{M_1^{-1}}_{f^*}(u^{k}-A^Tz^{k+1})$;}
\State{$u^{k+1}=u^{k}-A^Tz^{k+1}-y^{k+1}$;}
\EndFor
\end{algorithmic}
\end{algorithm}

Let us define the following generalized augmented Lagrangian for \eqref{equ: PLADMM}:
\begin{align}
\label{equ: generalized augmented Lagrangian}
    L(z,y,u)=g^*(z)+f^*(y)+\langle -A^Tz-y, M_1^{-1}u\rangle+\frac{1}{2}\|A^Tz+y\|^2_{M_1^{-1}}.
\end{align}
Inspired by \cite{wang2015global}, we use \eqref{equ: generalized augmented Lagrangian} as the Lyapunov function to establish convergence of Algorithm \ref{alg: Inexact PLADMM} and, equivalently, the convergence of Algorithm \ref{alg: Inexact PrePDHG}.
\begin{proposition}
\label{thm: Inexact PLADMM}
Under Assumption \ref{Assumption 1} and the transforms $u^{k}=M_1x^{k}$, $y^{k+1}=u^{k}-A^T{z}^{k}-u^{k+1}$, PrePDHG \eqref{equ: PrePDHG} is equivalent to \eqref{equ: PLADMM}, and iPrePDHG in Algorithm \ref{alg: Inexact PrePDHG} is equivalent to Algorithm \ref{alg: Inexact PLADMM}.
\end{proposition}
\begin{proof}
Set $u^{k}=M_1x^{k}$, $y^{k+1}=u^{k}-A^T{z}^{k}-u^{k+1}$. Then \eqref{equ: generalized Moreau} and \eqref{equ: PrePDHG} yield
\begin{align*}
    y^{k+1}&=M_1x^k-A^Tz^k-M_1x^{k+1}=\prox^{M_1^{-1}}_{f^*}(u^{k}-A^T{z}^{k}),
\end{align*}
and
\begin{align*}
    u^{k+1}&=u^{k}-A^T{z}^{k}-y^{k+1},\\
    {z}^{k+1}&=\prox^{M_2}_{g^*}({z}^{k}+M_2^{-1}AM_1^{-1}(-A^T{z}^{k}-y^{k+1}+u^{k+1})).
\end{align*}
If the $z$-update is performed first, then we arrive at \eqref{equ: PLADMM}.

In iPrePDHG or Algorithm \ref{alg: Inexact PrePDHG}, we are solving the $z$-subproblem of PrePDHG \eqref{equ: PrePDHG} approximately to the bounded relative error in Definition \ref{def: bounded relative error}. This is equivalent to doing the same to the $z$-subproblem of \eqref{equ: PLADMM}, which yields Algorithm \ref{alg: Inexact PLADMM}.
\end{proof}

We establish convergence under the following additional assumptions. 
\begin{assumption}
\label{Assumption 2}
\hfill
\begin{enumerate}
    \item $f(x)$ is $\mu_f-$strongly convex.
    \item $g^*(z)+f^*(-A^Tz)$ is coercive, i.e.,
    $
    \lim_{\|z\|\rightarrow \infty}g^*(z)+f^*(-A^Tz)=\infty.
    $
    \item $g^*(z)$ is a KŁ function.

\end{enumerate}
\end{assumption}

\begin{theorem}
\label{thm: Lyapunov}
Take Assumptions \ref{Assumption 1} and \ref{Assumption 2}. Choose any preconditioners $M_1, M_2$ and inner iteration number $p$ such that
\begin{align}
C_1&=\frac{1}{2}M_1^{-1}-\frac{\|M_1\|}{\mu_f^2}I_{n\times n}\succ 0,\label{equ: C_1}\\
C_2&=M_2-\frac{1}{2}AM_1^{-1}A^T-c(p)I_{m\times m}\succ 0,\label{equ: C_2}
\end{align}
where $c(p)$ depends on the $z$-subproblem iterator $S$ and $M_2$ (e.g., \eqref{equ: c} and \eqref{equ: BCD c}). Define $L^k\coloneqq L(z^{k},y^{k},u^{k})$. Then, Algorithm \ref{alg: Inexact PLADMM} satisfies the following sufficient descent and lower boundedness properties, respectively:
\begin{align}
    L^k-L^{k+1}&\geq \|y^k-y^{k+1}\|^2_{C_1}+\|z^{k}-z^{k+1}\|^2_{C_2},\label{equ: descent}\\
    L^k&\geq g^*(z^{\star})+f^*(-A^Tz^{\star})>-\infty.\label{equ: lower bounded}
\end{align}
\end{theorem}
\begin{proof}
Since the $z$-subproblem of Algorithm \ref{alg: Inexact PLADMM} is solved to the bounded relative error in Def. \ref{def: bounded relative error}, we have
\begin{align}
    \vz&\in \partial g^*(z^{k+1})+M_2(z^{k+1}-z^k-M_2^{-1}AM_1^{-1}(-A^Tz^k-y^k+u^k))+\varepsilon^{k+1},\label{equ: first optimality condition}
\end{align}
where $\varepsilon^{k+1}$ satisfies \eqref{equ: bounded}:
\begin{align}
\label{equ: vareps}
    \|\varepsilon^{k+1}\|\leq c(p)\|z^{k+1}-z^k\|.
\end{align}
The $y$ and $u$ updates produce
\begin{align}
    \vz&=\nabla f^*(y^{k+1})+M_1^{-1}(y^{k+1}-u^k+A^Tz^{k+1})=\nabla f^*(y^{k+1})-M_1^{-1}u^{k+1},\label{equ: key equality}\\
    u&^{k+1}=u^k-A^Tz^{k+1}-y^{k+1}.\label{equ: third optimality condition}
\end{align}
In order to show \eqref{equ: descent}, let us write
\begin{align*}
    g^*(z^k)&\geq g^*(z^{k+1})\\
    &+\langle M_2(z^k-z^{k+1})+AM_1^{-1}(-A^Tz^k-y^k+u^k)-\varepsilon^{k+1}, z^k-z^{k+1}\rangle,\\
    f^*(y^k)&\geq f^*(y^{k+1})+\langle M_1^{-1}u^{k+1}, y^k-y^{k+1}\rangle,
\end{align*}
Assembling these inequalities with \eqref{equ: vareps} gives us
\begin{align}
    L^k-L^{k+1}&\geq\|z^k-z^{k+1}\|^2_{M_2-c(p)I_{m\times m}}\nonumber\\
    &+\langle AM_1^{-1}(-A^Tz^k-y^k+u^k),  z^k-z^{k+1}\rangle+\langle M_1^{-1}u^{k+1}, y^k-y^{k+1}\rangle\nonumber\\
    &+\langle -A^Tz^k-y^k, M_1^{-1}u^k\rangle-\langle A^Tz^{k+1}-y^{k+1}, M_1^{-1}(u^k-A^Tz^{k+1}-y^{k+1})\rangle\nonumber\\
    &+\frac{1}{2}\|A^Tz^k+y^k\|^2_{M_1^{-1}}-\frac{1}{2}\|A^Tz^{k+1}+y^{k+1}\|^2_{M_1^{-1}}\nonumber\\
    &=\|z^k-z^{k+1}\|^2_{M_2-c(p)I_{m\times m}}\nonumber\\
    &+\langle AM_1^{-1}(-A^Tz^k-y^k),  z^k-z^{k+1}\rangle+\langle M_1^{-1}u^{k+1}, y^k-y^{k+1}\rangle\nonumber\label{equ: A}\tag{A}\\
    &+\langle -y^k, M_1^{-1}u^k\rangle-\langle -y^{k+1}, M_1^{-1}u^k\rangle\nonumber\label{equ: B}\tag{B}\\
    &+\frac{1}{2}\|A^Tz^k+y^k\|^2_{M_1^{-1}}-\frac{3}{2}\|A^Tz^{k+1}+y^{k+1}\|^2_{M_1^{-1}},\nonumber
\end{align}
where the terms in \eqref{equ: A} and \eqref{equ: B} simplify to
\begin{align}
\label{equ: descent 2}
     \langle AM_1^{-1}(-A^Tz^k-y^k),  z^k-z^{k+1}\rangle +\langle M_1^{-1}(-A^Tz^{k+1}-y^{k+1}), y^k-y^{k+1}\rangle.
\end{align}
Apply the following cosine rule on the two inner products above:
\[
\langle a-b, a-c \rangle_{M_1^{-1}}=\frac{1}{2}\|a-b\|^2_{M_1^{-1}}+\frac{1}{2}\|a-c\|^2_{M_1^{-1}}-\frac{1}{2}\|b-c\|_{M_1^{-1}}.
\]
Set $a=A^Tz^k, c=A^Tz^{k+1}$, and $b=-y^k$ to obtain
\begin{align}
     \langle AM_1^{-1}(-A^Tz^k-y^k),  z^k-z^{k+1}\rangle&=-\frac{1}{2}\|A^Tz^k+y^k\|^2_{M_1^{-1}}-\frac{1}{2}\|A^Tz^k-A^Tz^{k+1}\|^2_{M_1^{-1}}\nonumber\\
     &+\frac{1}{2}\|y^k+A^Tz^{k+1}\|^2_{M_1^{-1}}.\label{equ: cosine rule 1}
\end{align}
Set $a=y^{k+1}, c=y^k$, and $b=-A^Tz^{k+1}$ to obtain
\begin{align}
    \langle M_1^{-1}(-A^Tz^{k+1}-y^{k+1}), y^k-y^{k+1}\rangle&=\frac{1}{2}\|A^Tz^{k+1}+y^{k+1}\|^2_{M_1^{-1}}+\frac{1}{2}\|y^k-y^{k+1}\|_{M_1^{-1}}\nonumber\\
    &-\frac{1}{2}\|A^Tz^{k+1}+y^{k}\|^2_{M_1^{-1}}.\label{equ: cosine rule 2}
\end{align}
Combining \eqref{equ: descent 2}, \eqref{equ: cosine rule 1}, and \eqref{equ: cosine rule 2} yields
\begin{align}
    L^k-L^{k+1}&\geq\|z^k-z^{k+1}\|^2_{M_2-\frac{1}{2}AM_1^{-1}A^T-c(p)I_{m\times m}}+\|y^k-y^{k+1}\|^2_{\frac{1}{2}M_1^{-1}}\nonumber\\
    &-\|A^Tz^{k+1}+y^{k+1}\|^2_{M_1^{-1}}.\label{equ: descent 3}
\end{align}
Since $f$ is $\mu_f$-strongly convex,
we know that $\nabla f^*$ is $\frac{1}{\mu_f}-$Lipschitz continuous. Consequently,
\begin{align}
\|A^Tz^{k+1}+y^{k+1}\|^2_{M_1^{-1}}=\|u^k-u^{k+1}\|^2_{M_1^{-1}}
&\leq \frac{1}{\lambda_{\mathrm{min}}(M_1^{-1})}\|M_1^{-1}(u^k-u^{k+1})\|^2\nonumber\\
&\overset{\eqref{equ: key equality}}{\leq} \frac{\|M_1\|}{\mu^2_f}\|y^k-y^{k+1}\|^2.\label{equ: bound the dual}
\end{align}
Combining \eqref{equ: descent 3} and \eqref{equ: bound the dual} gives us \eqref{equ: descent}.

Now, to show \eqref{equ: lower bounded}, we use \eqref{equ: key equality} and smoothness of $f^*$ to get
\[
f^*(y^k)\geq f^*(-A^Tz^k)+\langle M_1^{-1}u^{k}, y^k+A^Tz^k\rangle - \frac{1}{2\mu_f} \|A^Tz^k+y^k\|^2.
\]
Hence, we arrive at
\begin{align}
    L^k&=g^*(z^k)+f^*(y^k)+\langle -A^Tz^k-y^k, M_1^{-1}u^k\rangle+\frac{1}{2}\|A^Tz^k+y^k\|^2_{M_1^{-1}}\nonumber\\
    &\geq g^*(z^k)+f^*(-A^Tz^k)+\frac{1}{2}\|A^Tz^k+y^k\|^2_{M_1^{-1}}-\frac{1}{2\mu_f}\|A^Tz^k+y^k\|^2.\label{equ: lower bounded 1}
\end{align}
Since $C_1\succ 0$ if and only if $\mu_f> \sqrt{2}\|M_1\|$, \eqref{equ: lower bounded} follows.
\end{proof}

Next, we provide a simple choice of $M_1, M_2, $ and $p$ that ensures the positive definiteness of $C_1$ and $C_2$ in Theorem \ref{thm: Lyapunov}.
\begin{lemma}
\label{lem: near ADMM choice}
In order to ensure \eqref{equ: C_1} and \eqref{equ: C_2}, it suffices to set $M_1=\frac{1}{\tau} I_{n\times n}$ where $\tau<\frac{1}{\sqrt{2}}\mu_f$, $M_2=\tau AA^T+\gamma I_{m\times m}$ with any $\gamma>0$, and $p$ is large enough such that $c(p)<\gamma$.
\end{lemma}
\begin{proof}
Since $M_1=\frac{1}{\tau}I_{n\times n}$, it is evident that $C_1\succ 0$ if and only if $\tau<\frac{1}{\sqrt{2}}\mu_f$.
With $M_1=\frac{1}{\tau}I_{n\times n}$ and $M_2 = \tau AA^T +I_{m\times m}$, we have 
\[
C_2 = \frac{1}{2}\tau AA^T +(\gamma-c(p))I_{m\times m},
\]
since $c(p)\propto \alpha^p$ for some $0<\alpha<1$, we know that there exists $p_0$ such that $C_2\succ 0$ for any $p\geq p_0$.
\end{proof}

We conclude this section by showing that  $(x^k, z^k)$ in Algorithm \ref{alg: Inexact PrePDHG} converges subsequentially to a primal-dual solution pair of \eqref{equ: PDHG problem} and \eqref{equ: PDHG dual problem}.

\begin{theorem}
\label{thm: subsequence convergence}
Let the assumptions in Theorem \ref{thm: Lyapunov} hold. Then, $(x^k, z^k)$ in Algorithm \ref{alg: Inexact PrePDHG} is bounded, and any cluster point of $\{x^k, z^k\}$ is a primal-dual solution pair of \eqref{equ: PDHG problem} and \eqref{equ: PDHG dual problem}.
\end{theorem}
\begin{proof}
According to Theorem \ref{thm: Inexact PLADMM}, it is sufficient to show that $\{M_1^{-1}u^k, z^k\}$ is bounded, and its cluster points are primal-dual solution pairs of \eqref{equ: PDHG problem}.

Since $L^k$ is nonincreasing, \eqref{equ: lower bounded 1} tells us that
\[
g^*(z^k)+f^*(-A^Tz^k)+\frac{1}{2}\|A^Tz^k+y^k\|^2_{M_1^{-1}}\leq L^0< +\infty.
\]
Since $g^*(z)+f^*(-A^Tz)$ is coercive, $\{z^k\}$ is bounded, and, by the boundedness of $\{A^Tz^k+y^k\}$, $\{y^k\}$ is also bounded. Furthermore, \eqref{equ: key equality} gives us
\[
\|M_1^{-1}(u^k-u^0)\|\leq \frac{1}{\mu_f}\|y^k-y^0\|.
\]
Therefore, $\{M_1^{-1}u^k\}$ is bounded, too.

Let $(z^c, y^c, u^c)$ be a cluster point of $\{z^k, y^k, u^k\}$. We shall show $(z^c, y^c, u^c)$ is a saddle point of $L(z, y, u)$, i.e.,
\begin{align}
\label{equ: saddle point}
\vz\in \partial L(z^c, y^c, u^c),
\end{align}
or equivalently,
\begin{align*}
    \vz &\in \partial g^*(z^c)-AM_1^{-1}u^c,\\
    \vz &= \nabla f^*(y^c)-M_1^{-1}u^c,\\
    \vz &= A^Tz^c+y^c,
\end{align*}
which ensures $(M_1^{-1}u^c, z^c)$ to be a primal-dual solution pair of \eqref{equ: PDHG problem}.

In order to show \eqref{equ: saddle point}, we first notice that \eqref{equ: generalized augmented Lagrangian} gives
\begin{align*}
    \partial_x L(z^{k+1}, y^{k+1}, u^{k+1})&=\partial g^*(z^{k+1})-AM_1^{-1}u^{k+1}+AM_1^{-1}(A^Tz^{k+1}+y^{k+1}),\\
    \nabla_y L(z^{k+1}, y^{k+1}, u^{k+1})&=\nabla f^*(y^{k+1})-M_1^{-1}u^{k+1}+M_1^{-1}(A^Tz^{k+1}+y^{k+1}),\\
    \nabla_u L(z^{k+1}, y^{k+1}, u^{k+1})&=M_1^{-1}(-A^Tz^{k+1}-y^{k+1}).
\end{align*}
Comparing these with the optimality conditions \eqref{equ: first optimality condition}, \eqref{equ: key equality}, and \eqref{equ: third optimality condition}, we have
\[
d^{k+1}=(d^{k+1}_z, d^{k+1}_y, d^{k+1}_u)\in\partial L(z^{k+1}, y^{k+1}, u^{k+1}),
\]
where
\begin{align*}
    d^{k+1}_z&=M_2(z^k-z^{k+1})+2AM_1^{-1}(u^k-u^{k+1})-AM_1^{-1}(u^{k-1}-u^k)-\varepsilon^{k+1},\\
    d^{k+1}_y&=M_1^{-1}(u^k-u^{k+1}),\\
    d^{k+1}_u&=M_1^{-1}(u^{k+1}-u^k).
\end{align*}
Since \eqref{equ: descent} and \eqref{equ: lower bounded} imply $z^k-z^{k+1}, y^k-y^{k+1}\rightarrow \vz$,  \eqref{equ: key equality} gives $u^k-u^{k+1}\rightarrow \vz$. Combine these with \eqref{equ: bounded}, we have $d^k\rightarrow \vz$.

Finally, let us take a subsequence $\{z^{k_s}, y^{k_s}, u^{k_s}\}\rightarrow (z^c, y^c, u^c)$. Since $d^{k_s}\rightarrow \vz$ as $s\rightarrow +\infty$, \cite[Def. 8.3]{rockafellar2009variational} and \cite[Prop. 8.12]{rockafellar2009variational} yield \eqref{equ: saddle point}, which tells us that $(M_1^{-1}u^c, z^c)$ is a primal-dual solution pair of \eqref{equ: PDHG problem}.
\end{proof}

We can show that the whole sequence $(x^k, z^k)$ in Algorithm \ref{alg: Inexact PrePDHG}  converges. Since the proof consists of a standard technique of using the KŁ property in Assumption \ref{Assumption 2}, which is not  relevant to the main idea of this subsection, we leave it to Appendix \ref{App: KL}.
\begin{theorem}
\label{thm: sequence convergence}
Let the assumptions in Theorem \ref{thm: Lyapunov} hold. Then,  $\{x^k, z^k\}$ in Algorithm \ref{alg: Inexact PrePDHG} converges to a primal-dual solution pair of \eqref{equ: PDHG problem}.
\end{theorem}

\begin{proof}
See Appendix \ref{App: KL}.
\end{proof}

\section{Numerical experiments}
\label{Section: Numerical experiments}

In this section, we compare our iPrePDHG (Algorithm \ref{alg: Inexact PrePDHG}) with (original) PDHG \eqref{equ: PDHG} and diagonally-preconditioned PDHG (DP-PDHG) \cite{pock2011diagonal}. We consider four popular applications of PDHG: TV-$\text{L}^1$ denoising, graph cuts, estimation of earth mover's distance, and CT reconstruction. 

For the preconditioners $M_1$ and $M_2$ in iPrePDHG, We will choose $M_1=\frac{1}{\tau}I_{n\times n}$ and $M_2 = \tau AA^T$, which corresponds to ADMM and $M_2$ is optimal (see subsection \ref{Subsection: Convergence of PrePDHG and Choice of Preconditioners}). The number of inner loops $p$ is taken from $\{1,2,3\}$. Although these choices may be more aggressive than what's presented in Lemma \ref{lem: near ADMM choice} and $f$ may not be strongly convex in our experiments, we still observe significant speedup compared to other algorithms.

When we write these examples in the form of \eqref{equ: PDHG problem}, the matrix $A$ (or a part of $A$) is one of the following operators:

\begin{enumerate}[label=\textbf{\ Case \arabic*:},itemindent=*,listparindent=2em]

     \item  2D discrete gradient operator $D: \R^{M\times N}\rightarrow \R^{2M\times N}$:

    \noindent For images of  size $M\times N$ and grid step size  $h$, we have.
\begin{align*}
(D u)_{i,j}=\begin{pmatrix}(D u)^1_{i,j} \\ (D u)_{i,j}^2\end{pmatrix},
\end{align*}
where
\begin{align*}
    (D u)_{i,j}^1&=
   \begin{cases}
  \frac{1}{h}(u_{i+1,j}-u_{i,j}) &\mbox{if $i<M$},\\
   0 &\mbox{if $i=M$},
   \end{cases}
   \\
    (D u)_{i,j}^2&=
   \begin{cases}
  \frac{1}{h}(u_{i,j+1}-u_{i,j}) &\mbox{if $j<N$},\\
   0 &\mbox{if $j=N$}.
   \end{cases}
\end{align*}
 \item  Weighted gradient operator $D_w:\R^{M\times N}\rightarrow \R^{ 2M\times N}$:
\[
D_w = \text{diag}(w)D,
\]
where $w\in(\R^+)^{2MN}$ is a weight vector.
 \item  2D discrete divergence operator: div$:\R^{2M\times N}\rightarrow \R^{ M\times N}$ given by
\begin{align}\label{div expression}
\text{div}(p)_{i,j} = h(p_{i,j}^1 - p_{i-1,j}^1+p_{i,j}^2-p_{i,j-1}^2),
\end{align}
where $p=(p^1, p^2)^T\in\R^{2M\times N}$, $p^1_{0,j} = p^1_{M,j} = 0$ and $p^2_{i,0} = p^2_{i,N} = 0$ for $i = 1,...,M$, $j = 1,...,N$.

\end{enumerate}

To take advantages of the finite-difference structure of these operators, we let $S$ be the iterator of cyclic proximal BCD in Algorithm \ref{alg: Inexact PrePDHG}. We split $\{1,2,...m\}$ into $2$ blocks (for case 3) or $4$ blocks (for cases 1 and 2), 
which are inspired by the popular red-black ordering \cite{saad2003iterative} for solving sparse linear system.

According to Theorem \ref{thm: BCD finite inner loops}, running finitely many epochs of cyclic proximal BCD gives us a bounded relative error in Def. \ref{def: bounded relative error}. We expect that this solver brings faster overall convergence. Specifically, when $g^*$ is linear (or equivalently, $g$ is a $\delta$ function), the $z$-subproblem in PrePDHG reduces to a linear system with a structured sparse matrix $AA^T$. Therefore, Gradient Descent amounts to the Richardson method \cite{richardson1911ix, saad2003iterative}, and cyclic proximal BCD is equivalent to the Gauss-Seidel method \cite{gauss1903werke, saad2003iterative}. 
The following two claims tell us that $S$ in Algorithm \ref{alg: Inexact PrePDHG} has a closed form, so Algorithm \ref{alg: Inexact PrePDHG} is easy to implement. Furthermore, each execution of $S$ can use parallel computing.
\begin{figure}[htbp]
\centering
\begin{minipage}[t]{0.48\textwidth}
\centering
\includegraphics[width=6cm]{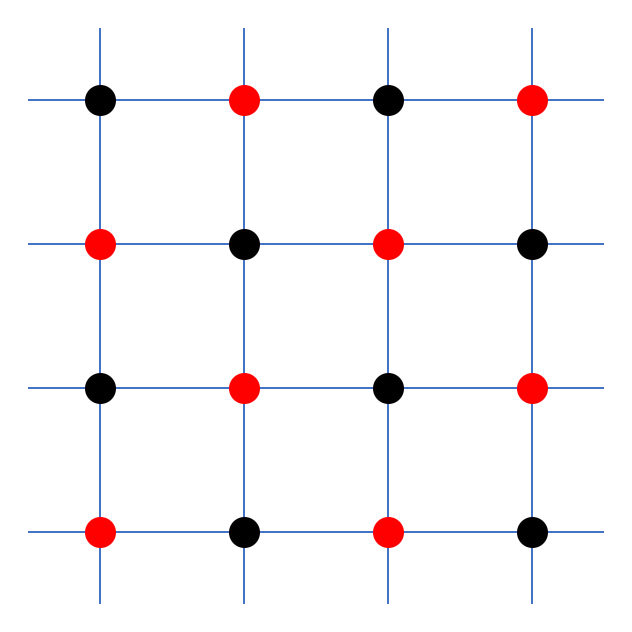}
\caption{two-block ordering in Claim \ref{claim1}}
\label{2block}
\end{minipage}
\begin{minipage}[t]{0.48\textwidth}
\centering
\includegraphics[width=6cm]{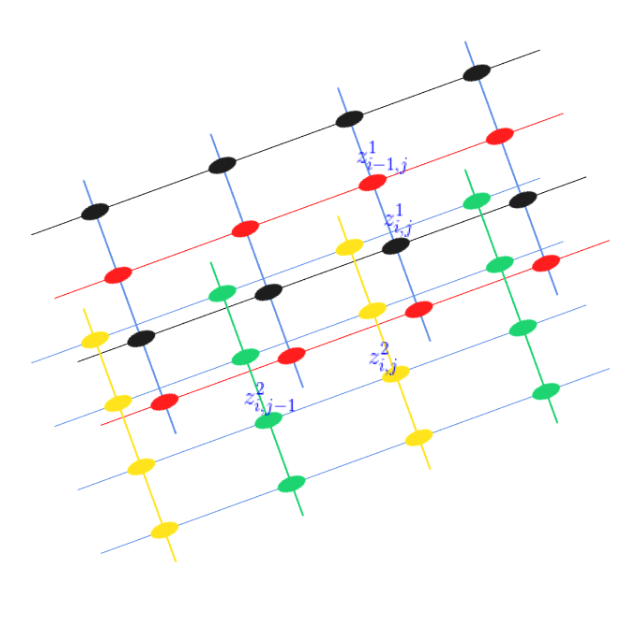}
\caption{four-block ordering in Claim \ref{claim2}}
\label{4block}
\end{minipage}
\end{figure}

\begin{claim}
\label{claim1}
When $A=\rm div$ (i.e. $A^T=-D$) and $M_2=\tau AA^T$, for $z\in {\mathbb{R}}^{M\times N}$, we separate $z$ into two block $z_b$, $z_r$ where
\[
z_b:=\{ z_{i,j}\,|\,i+j\text{ is even}\}, ~ z_r:=\{ z_{i,j}\,|\,i+j\text{ is odd}\},\]
for $1\leq i\leq M$, $1\leq j\leq N$. If $g(z)=\Sigma_{i,j}g_{i,j}(z_{i,j})$ and $prox_{\lambda g^*_{i,j}}$ have closed-form solutions for all $1\leq i\leq M$, $1\leq j\leq N$ and $\lambda > 0$, then $S$ as the iterator of cyclic proximal BCD in Algorithm \ref{alg: Inexact PrePDHG} has a closed form and computing $S$ is parallelizable.
\end{claim}
 \begin{proof}
As illustrated in Fig. \ref{2block}, every black node is connected to its neighbor red nodes, so we can update all the coordinates corresponding to the black nodes in parallel, while those corresponding to the red nodes are fixed, and vice versa. See Appendix \ref{App: Smart Ordering} for a complete explanation.
 \end{proof}
\begin{claim}
\label{claim2}
   When $A= D$ or $A=D_w$ (i.e. $A^T=-\rm div$ or $A^T=-\rm div\, diag(w)$) and $M_2=\tau AA^T$, for $z=(z^1, z^2)^T \in {\mathbb{R}}^{2M\times N}$, we separate $z$ into four blocks $z_b$, $z_r$, $z_y$ and $z_g$, where
  \begin{align*}
  z_b&=\{ z^1_{i,j}\,|\,\text{$i$ is odd}\},
  \quad z_r=\{z^1_{i,j}\,|\,\text{$i$ is even}\},\nonumber\\
  z_y&=\{z^2_{i,j}\,|\,\text{$j$ is odd}\},
  \quad z_g=\{z^2_{i,j}\,|\,\text{$j$ is even}\},\nonumber
  \end{align*}
  for $1\leq i\leq M$, $1\leq j\leq N$. If $g(z)=\Sigma_{i,j}g_{i,j}(z_{i,j})$ and $prox_{\lambda g^*_{i,j}}$ have closed-form solutions for all $1\leq i\leq M$, $1\leq j\leq N$ and $\lambda > 0$, then $S$ as the iterator of cyclic proximal BCD in Algorithm \ref{alg: Inexact PrePDHG} has a closed form and computing $S$ is parallelizable. 
\end{claim}
 \begin{proof}
 In Figure \ref{4block}, the 4 blocks are in 4 different colors. The coordinates corresponding to nodes of the same color can be updated in parallel, while the rest are fixed. See Appendix \ref{App: Smart Ordering} for details.
 \end{proof}

In Table \ref{Table: TV-L1}, Table \ref{Table: seg}, Fig. \ref{figure-EMDtime}, and Table \ref{Table: CT}, PDHG denotes original PDHG in \eqref{equ: PDHG} without any preconditioning; DP-PDHG denotes the diagonally-preconditioned PDHG in \cite{pock2011diagonal}, PrePDHG denotes Preconditioned PDHG in \eqref{equ: PrePDHG} where the $(k+1)$th $z$-subproblem is solved until $\frac{\|z^k-z^{k+1}\|_2}{\max\{1,\|z^{k+1}\|_2\}}< 10^{-5}$ using the TFOCS  \cite{becker2011templates} implementation of FISTA with restart; iPrePDHG (S=BCD) and iPrePDHG (S=FISTA) denote our iPrePDHG in Algorithm \ref{alg: Inexact PrePDHG} with the iterator $S$ being
cyclic proximal BCD or FISTA with restart, respectly. 
All the experiments were performed on MATLAB R2018a on a MacBook Pro with a 2.5 GHz Intel i7 processor and 16GB of 2133MHz LPDDR3 memory.

A comparison between PDHG and DP-PDHG is presented in \cite{pock2011diagonal} on TV-$\text{L}^1$ denoising and graph cuts, and in \cite{sidky2012convex} on CT reconstruction.  A PDHG algorithm  is proposed to estimate earth mover's distance (or optimal transport) in \cite{li2017parallel}. In order to provide a direct comparison, we use their problem formulations.

\subsection{Total variation based image denoising}\label{subsec-tvl1}
The following problem is known as the (discrete) TV-$\text{L}^1$ model for image denoising:
\begin{equation*}
\begin{array}{ll}
\mathop{\mathrm{minimize}}_u & \Phi(u) = \|D u\|_{1}+\lambda \|u-b\|_{1},\\
\end{array}
\end{equation*}
where $D$ is the 2D discrete gradient operator with $h=1$, $b\in\R^{M\times N}$ is a noisy input image, and $\lambda$ is a regularization parameter. In our experiment we input a $1024\times 1024$ image with noise level 0.15 and set $\lambda=1$; see Fig. \ref{Noisy image}. We run the algorithms until $\delta^k\coloneqq\frac{|\Phi^k-\Phi^{\star}|}{|\Phi^{\star}|}< 10^{-6}$, where $\Phi^k$ is the objective value at $k$th iteration and $\Phi^*$ is the optimal objective value obtained by calling CVX \cite{cvx, gb08}.

Observed performance is summarized in Table \ref{Table: TV-L1}, where the best results for $\tau \in \{10,1,0.1,0.01,0.001\}$ { and $p\in\{1,2,3\}$} are presented.  Our iPrePDHG (S=BCD) is significantly faster than the other three algorithms.

Remarkably, our algorithm uses fewer outer iterations than PrePDHG under the stopping criterion  $\frac{\|z^k-z^{k+1}\|_2}{\max\{1,\|z^{k+1}\|_2\}}< 10^{-5}$, as this kind of stopping criteria
may become looser as $z^k$ is closer to $z^{\star}$. In this example, $\frac{\|z^k-z^{k+1}\|_2}{\max\{1,\|z^{k+1}\|_2\}}< 10^{-5}$ only requires $1$ inner iteration of FISTA  when Outer Iter $\geq368$, while as high as $228$ inner iterations on average during the first $100$ outer iterations. In comparison, our algorithm uses fewer outer iterations while each of them also costs less.

In addition, the diagonal preconditioner given in \cite{pock2011diagonal} appears to help very little when $A=D$. In fact, $M_1=\text{diag}(\Sigma_{i}|A_{i,j}|)$ will be $4I_n$ and $M_2=\text{diag}(\Sigma_{j}|A_{i,j}|)$ will be $2I_m$ if we ignore the Neumann boundary condition. 
Therefore, DP-PDHG performs even worse than PDHG.

\begin{table}[H]
\label{Table: TV-L1}
\begin{center}
\begin{small}
\begin{tabular}{c|c|c|c}

Method & Parameters  & Outer Iter& Runtime(s)\\
\hline\hline
\multirow{1}{*}{PDHG} &$\tau=0.01, M_1=\frac{1}{\tau}I_{n}, M_2 = \tau\|D\|^2I_{m}$  & 2990 & 114.2576\\
\hline
DP-PDHG & $M_1=\text{diag}(\Sigma_i |D_{i,j}|), M_2=\text{diag}(\Sigma_j|D_{i,j}|)$ & 8856& 329.7890\\
\hline
\multirow{1}{*}{PrePDHG} &\multirow{2}{*}{$\tau=0.1, M_1 =\frac{1}{\tau}I_{n}, M_2=\tau D D^T$} & \multirow{2}{*}{963} &  \multirow{2}{*}{5706.2837} \\(ADMM) & & &\\
\hline
\multirow{1}{*}{iPrePDHG} &\multirow{2}{*}{$\tau=0.01, M_1 = \frac{1}{\tau}I_{n}, M_2=\tau D D^T$, $p=1$}& \multirow{2}{*}{\textbf{541}} & \multirow{2}{*}{\textbf{26.2704}}  \\(S=BCD) & & &\\
\hline
\end{tabular}\\
\end{small}
\end{center}
\caption{TV-$L^1$ denoising test. PDHG is original PDHG. DP-PDHG uses diagonal preconditioning. PrePDHG uses non-diagonal preconditioning. iPrePDHG (S=BCD) is our algorithm that uses both non-diagonal preconditioning and an iterator $S$ instead of solving the $z$-subproblem.}

\end{table}

\begin{figure}[htbp]
\centering
\begin{minipage}[t]{0.48\textwidth}
\centering
\includegraphics[width=6cm]{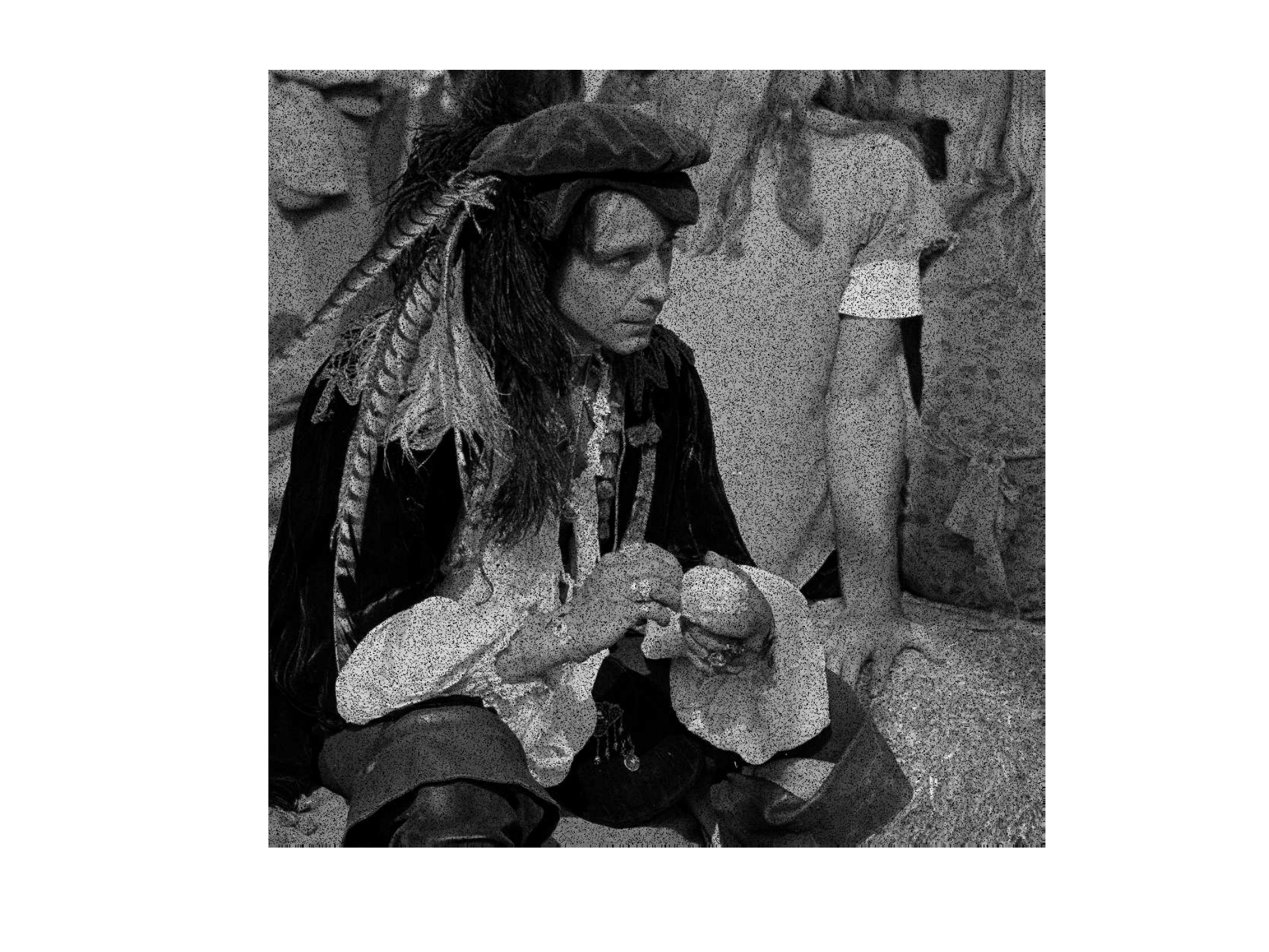}
\caption{Noisy image}
\label{Noisy image}
\end{minipage}
\begin{minipage}[t]{0.48\textwidth}
\centering
\includegraphics[width=6cm]{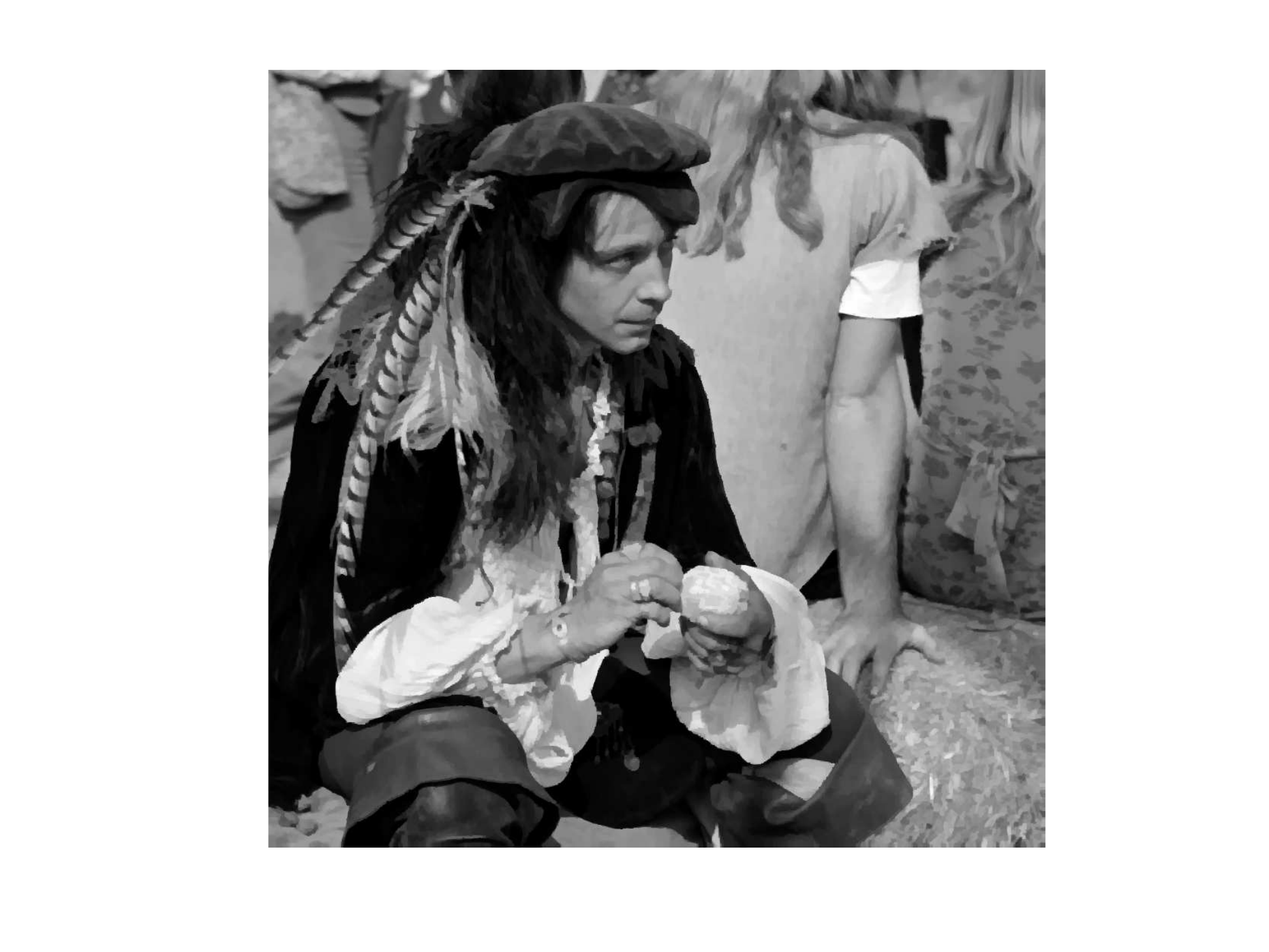}
\caption{Denoising by iPrePDHG (S=BCD)}
\label{Denoised image}
\end{minipage}
\end{figure}

\subsection{Graph cuts}
The total-variation-based graph cut model involves minimizing a weighted TV energy:
\begin{equation*}
\begin{array}{ll}
\mbox{minimize}& \|D_w u\|_{1}+\langle u,\omega^u\rangle\\
\mbox{subject to}& 0\leq u\leq 1, \\
\end{array}
\end{equation*}
where $w^u\in\R^{M\times N}$ is a vector of unary weights, $w^b\in\R^{2MN}$ is a vector of binary weights, and $D_w = \text{diag}(w^b)D$ for $D$ being the 2D discrete gradient operator with $h=1$.  Specifically, we have $w_{i,j}^u=\alpha(\|I_{i,j}-\mu_f\|^2-\|I_{i,j}-\mu_b\|^2)$, $w_{i,j}^{b,1}=\text{exp}(-\beta |I_{i+1,j}-I_{i,j}|)$, and  $w_{i,j}^{b,2}=\text{exp}(-\beta|I_{i,j+1}-I_{i,j}|)$. In our experiment, the image has a size $660\times 720$, and we set  $\alpha=1/2$, $\beta=10$, $\mu_f=[0;0;1]$ (for the blue foreground) and $\mu_b=[0;1;0]$ (for the green background). We run all algorithms until $\delta^k\coloneqq\frac{|\Phi^k-\Phi^{\star}|}{|\Phi^{\star}|}< 10^{-8}$, where $\Phi^k$ is the objective value at the $k$th iteration and $\Phi^*$ is the optimal objective value obtained by running CVX.

The best results of $\tau \in \{10,1,0.1,0.01,0.001\}$  and $p\in\{1,2,3\}$ are summarized in Table \ref{Table: seg}, where we can see that our iPrePDHG (S=BCD) is the fastest. It is also worth mentioning that its number of outer iterations is close to that of PrePDHG, which solves $z$-subproblem much more accurately.

\begin{table}[H]
\label{Table: seg}
\begin{center}
\begin{small}
\begin{tabular}{c|c|c|c}
Method & Parameters  & Outer Iter& Runtime(s)\\
\hline\hline
\multirow{1}{*}{PDHG} &$\tau=1, M_1=\frac{1}{\tau}I_{n}, M_2 = \tau\|D_w\|^2I_{m}$  & 5529 & 140.5777\\
\hline
\multirow{2}{*}{DP-PDHG} & $M_1=\text{diag}(\Sigma_i |{D_w}_{i,j}|),$ & \multirow{2}{*}{3571}& \multirow{2}{*}{104.5392}\\
& $M_2=\text{diag}(\Sigma_j|{D_w}_{i,j}|)$ & &\\
\hline
\multirow{1}{*}{PrePDHG} &\multirow{2}{*}{$\tau=10, M_1 =\frac{1}{\tau}I_{n}, M_2=\tau D_w D_w^T$} & \multirow{2}{*}{282} & \multirow{2}{*}{938.3787} \\(ADMM) & & & \\
\hline
\multirow{1}{*}{iPrePDHG} &\multirow{2}{*}{$\tau=10, M_1 =\frac{1}{\tau}I_{n}, M_2=\tau D_w D_w^T$, $p=2$}& \multirow{2}{*}{\textbf{411}} & \multirow{2}{*}{\textbf{14.9663} } \\ (S=BCD) & & &\\
\hline
\end{tabular}\\
\end{small}
\end{center}
\caption{Graph cut test}
\end{table}

\begin{figure}[htbp]
\centering
\begin{minipage}[t]{0.48\textwidth}
\centering
\includegraphics[width=6cm]{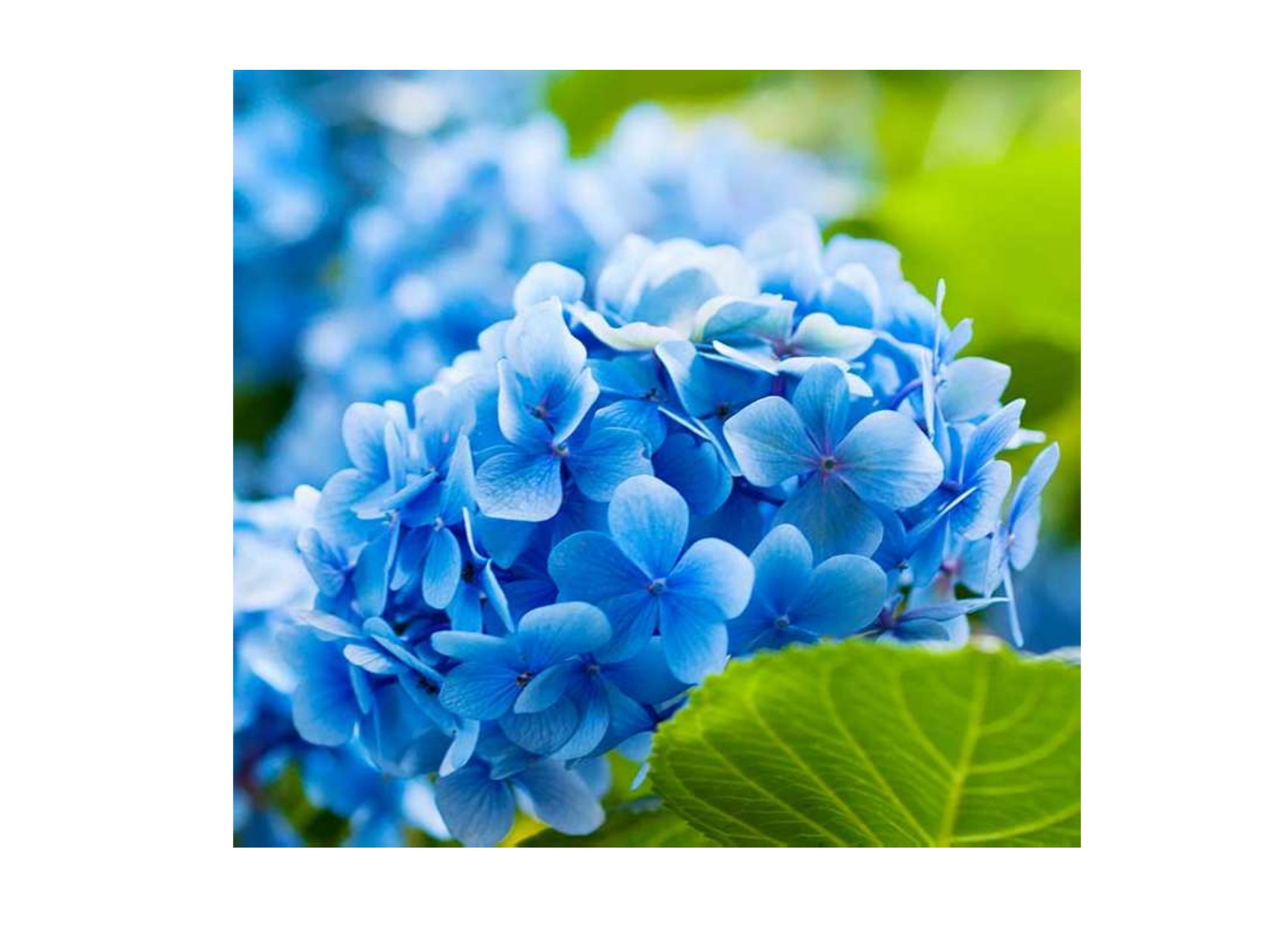}
\caption{Input image}
\label{Input image}
\end{minipage}
\begin{minipage}[t]{0.48\textwidth}
\centering
\includegraphics[width=6cm]{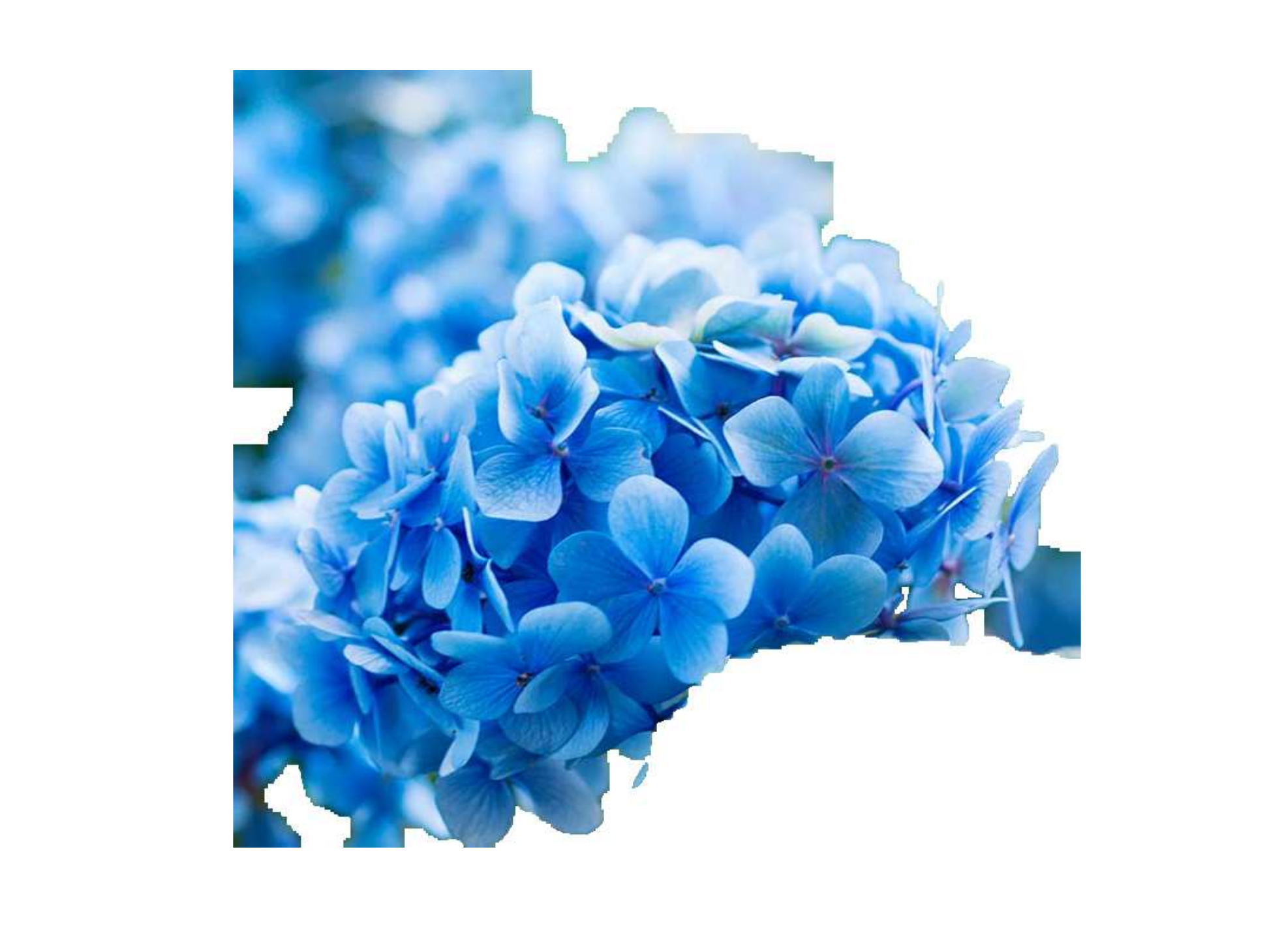}
\caption{Graph cut by iPrePDHG (S=BCD)}
\label{Graph cut}
\end{minipage}
\end{figure}

\subsection{Earth mover's distance}
Earth mover's distance is useful in image processing, computer vision, and statistics \cite{levina2001earth, metivier2016measuring, pele2009fast}. A recent method \cite{li2017parallel} to compute earth mover's distance is based on
\begin{equation}\label{equ: EMD formulation}
\begin{array}{ll}
\mbox{minimize}& \|m\|_{1,2}\\
\mbox{subject to}& \text{div}(m)+\rho^1-\rho^0=0,
\end{array}
\end{equation}
where $m\in\R^{2M\times N}$ is the sought flux vector on the $M\times N$ grid, and $\rho^0, \rho^1$ represents two mass distributions on the $M\times N$ grid. The setting in our experiment here is the same with that in \cite{li2017parallel}, i.e. $M=N=256$, $h=\frac{N-1}{4}$, and for $\rho^0$ and $\rho^1$ see Fig. \ref{figure-cat}.

Since the iterates $m^k$ may not satisfy the linear constraint, the objective $\Phi(m)=I_{\{m|div(m)=\rho^0-\rho^1\}}+\|m\|_{1,2}$ is not comparable. Instead, we compare $\|m^k\|_{1,2}$ and the constraint violation until $k=100000$ outer iterations in Fig. \ref{fig: EMD}, where we set $\tau=3\times 10^{-6}$ as in \cite{li2017parallel}, and $\sigma=\frac{1}{\tau \|\text{div}\|^2}$.
In Fig. \ref{fig: EMD}, we can see that our algorithm provides much lower constraint violation and much more faithful earth mover's distance $\|m\|_{1,2}$.  Fig. \ref{figure-cat} shows the solution obtained by our iPrePDHG (S=BCD), where $m$ is the flux that moves the standing cat $\rho^1$ into the crouching cat $\rho^0$. DP-PDHG and PrePDHG are extremely slow in this example. Similar to \ref{subsec-tvl1}, when $A=\text{div}$, the diagonal preconditioners proposed in \cite{pock2011diagonal} are approximately equivalent to fixed constant parameters $\tau=\frac{1}{2h}$, $\sigma=\frac{1}{4h}$ and they lead to extremely slow convergence. As for PrePDHG, it suffers from the high cost per outer iteration. 

It is worth mentioning that unlike \cite{li2017parallel}, the algorithms in our experiments are not parallelized. On the other hand, in our iPrePDHG (S=BCD), iterator $S$ 
can be parallelized (which we did not implement). Therefore, one can expect a further speedup by a parallel implementation.

 \begin{figure}[htbp!]
 \label{fig: EMD}
 \vskip 0.2in
 \begin{center}
 \centerline{\includegraphics[width=5 in]{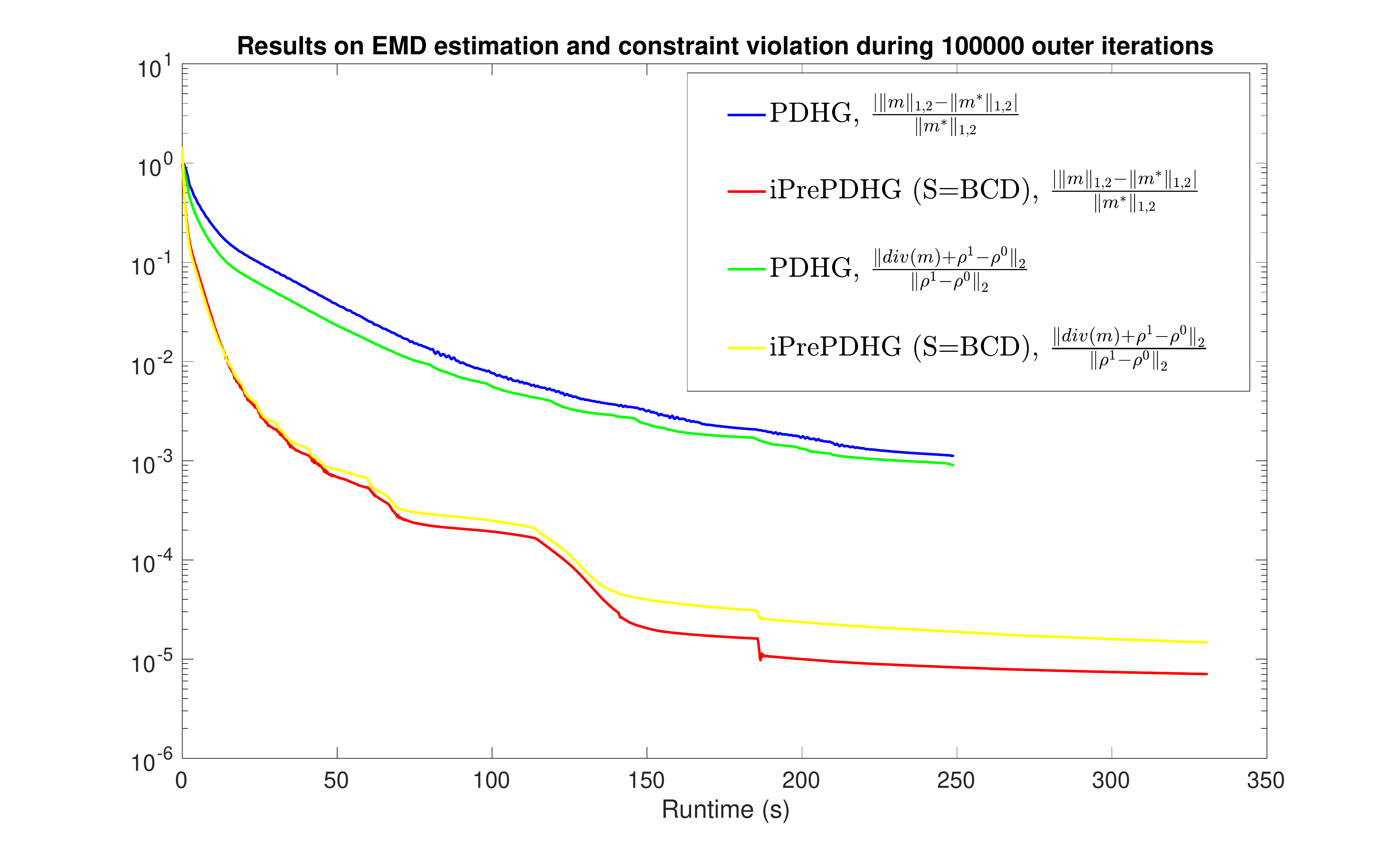}}
 \caption{For PDHG, $\tau=3\times 10^{-6}$, $\sigma=\frac{1}{\tau \|\rm div\|^2}$; For iPrePDHG (S=BCD), $\tau=3\times 10^{-6}$, $M_1=\tau^{-1}I_n$, $M_2=\tau \rm div \rm div^T$, $p=2$. $\|m^*|\|_{1,2}$ is obtained by calling CVX.}
 \label{figure-EMDtime}
 \end{center}
 \vskip -0.2in
 \end{figure}

 \begin{figure}[htbp!]
 \vskip 0.2in \begin{center}
 \centerline{\includegraphics[width=4 in]{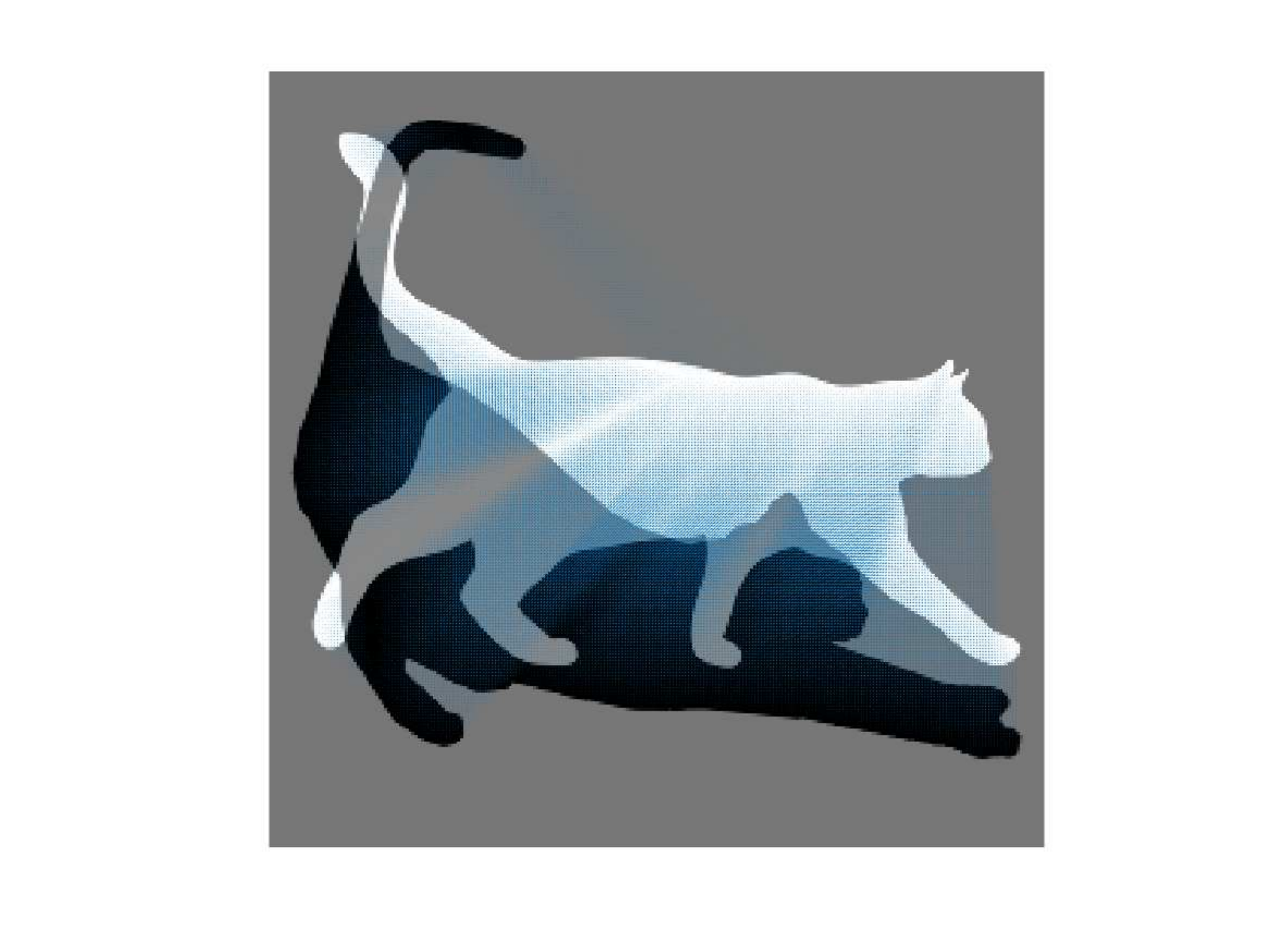}}
 \caption{ $\rho^0$, $\rho^1$ are the white standing cat and the black crouching cat, respectively. Both images are $256\times 256$, and the earth mover's distance between $\rho^0$ and $\rho^1$ is 0.6718.}
 \label{figure-cat}
 \end{center}
 \vskip -0.2in
 \end{figure}

\subsection{CT reconstruction}\label{subsec: CT}
We test solving the following optimization problem for CT image reconstruction:
\begin{equation}\label{equ: CT reconstruction}
\begin{array}{ll}
\mbox{minimize}& \Phi(u)=\frac{1}{2}\|Ru-b\|^2_{2}+\lambda\|Du\|_{1},
\end{array}
\end{equation}
where $R\in \ \R^{13032\times65536}$ is a system matrix for 2D fan-beam CT with a curved detector, $b=Ru_{\text{true}}\in \R^{13032}$ is a vector of line-integration values, and we want to reconstruct $u_{\text{true}}\in\R^{MN}$, where $M=N=256$. $D$ is the 2D discrete gradient operator with $h=1$, and $\lambda=1$ is a regularization parameter. By using the \textit{fancurvedtomo} function from the AIR Tools II \cite{hansen2018air} package, we generate a test problem where the projection angles are $\ang{0}, \ang{10}, \dots, \ang{350}$, and for all the other input parameters we use the default values.

Following \cite{sidky2012convex}, we formulate the problem \eqref{equ: CT reconstruction}  in the form of \eqref{equ: PDHG problem} by taking
\begin{align}
    g\begin{pmatrix} p\\q\end{pmatrix}= \frac{1}{2}\|p-b\|_2^2+\lambda\|q\|_1,\quad f(u)=0, \quad A=\begin{pmatrix} R\\D\end{pmatrix},
\end{align}
By using this formulation, we avoids inverting the matrices $R$ and $D$.

Since the block structure of $AA^T$ is rather complicated, if we naively choose $M_1=\frac{1}{\tau}I_n$ and $M_2=\tau AA^T$ like in the previous three experiments, it becomes hard to find a fast subproblem solver for the $z$-subproblem. 
In Table \ref{Table: CT}, we report a TFOCS implementation of FISTA for solving the $z$-subproblem and the overall convergence is very slow. 

Instead, we propose to choose
\begin{align}\label{equ: CT preconditioner choice}M_1=\frac{2}{\tau}I_{n},\quad  M_2=\begin{pmatrix}
    \tau\|R\|^2I_{m-2n} & 0\\
    0 & \tau DD^T
    \end{pmatrix}
\end{align}
or
\begin{align}
    \label{equ: CT preconditioner choice 1}
    M_1=\text{diag}(\Sigma_i |{R}_{i,j}|)+\frac{1}{\tau}I_{n}, \quad M_2 = \begin{pmatrix}
    \text{diag}(\Sigma_j|{R}_{i,j}|) & 0\\
    0 & \tau DD^T
    \end{pmatrix}.
\end{align}
These choices satisfy \eqref{equ: M}, and have simple block structures, a fixed epoch of $S$ as cyclic proximal BCD iterators gives fast overall convergence. Note that 
\eqref{equ: CT preconditioner choice 1} is a little slower but avoids the need of estimating $\|R\|$.

We summarize the numerical results in Table \ref{Table: CT}. All the algorithms are executed until $\delta^k\coloneqq\frac{|\Phi^k-\Phi^{\star}|}{|\Phi^{\star}|}< 10^{-4}$, where $\Phi^k$ is the objective value at the $k$th iteration and $\Phi^*$ is the optimal objective value obtained by calling CVX. The best results of $\tau\in\{10,1,0.1,0.01,0.001\}$ and $p\in\{1,2,3\}$ are summarized in Table \ref{Table: CT}, for iPrePDHG (S=FISTA) with $M_2=\tau AA^T$, the result for $p=100$ is also reported (here we use the TFOCS implementation of FISTA). 

\begin{table}[htbp]
\label{Table: CT}
\begin{center}
\begin{small}
\begin{tabular}{c|c|c|c}
Method & Parameters  & Outer Iter & Runtime(s)\\
\hline
\hline
\multirow{1}{*}{PDHG} &$\tau=0.001, M_1=\frac{1}{\tau}I_{n}, M_2 = \tau\|A\|^2I_{m}$  & 364366 & 3663.0348 \\
\hline
\multirow{2}{*}{DP-PDHG} & $M_1=\text{diag}(\Sigma_i |{A}_{i,j}|),$ &\multirow{2}{*}{70783} & \multirow{2}{*}{713.9865} \\
& $M_2=\text{diag}(\Sigma_j|{A}_{i,j}|)$ & &\\
\hline
\multirow{1}{*}{PrePDHG} &\multirow{2}{*}{$\tau=0.01, M_1 = \frac{1}{\tau}I_{n}, M_2=\tau A A^T$} & \multirow{2}{*}{-} & \multirow{2}{*}{$>10^4$} \\(ADMM) & & &\\
\hline
\multirow{1}{*}{iPrePDHG} &$\tau=0.001, M_1 = \frac{1}{\tau}I_{n}$,& \multirow{2}{*}{-}  & \multirow{2}{*}{$>10^4$}  \\(S=FISTA)
& $M_2=\tau A A^T, p=1, 2, \text{or}\,\, 3$ & &\\
\hline
\multirow{1}{*}{iPrePDHG} &$\tau=0.01, M_1 = \frac{1}{\tau}I_{n}$,& \multirow{2}{*}{-}  & \multirow{2}{*}{$>10^4$}  \\(S=FISTA)
& $M_2=\tau A A^T, p=100$ & &\\
\hline
\multirow{2}{*}{iPrePDHG} &$\tau=0.01,M_1=\frac{2}{\tau}I_{n},$  & \multirow{3}{*}{\textbf{587}} & \multirow{3}{*}{\textbf{7.5365}}\\

        (S=BCD)             & $M_2=\begin{pmatrix}
    \tau \|R\|^2 & 0\\
    0 & \tau DD^T
    \end{pmatrix}, p=2$ &  &     \\
\hline
\multirow{2}{*}{iPrePDHG} &$\tau=0.01,M_1=\text{diag}(\Sigma_i |{R}_{i,j}|)+\frac{1}{\tau}I_{n},$  & \multirow{3}{*}{\textbf{858}} & \multirow{3}{*}{\textbf{10.3517}}\\

                  (S=BCD)   & $M_2=\begin{pmatrix}
    \text{diag}(\Sigma_j|{R}_{i,j}|) & 0\\
    0 & \tau DD^T
    \end{pmatrix}, p=2$ &  &     \\
\hline
\end{tabular}\\
\end{small}
\end{center}
\caption{CT reconstruction}
\end{table}


\section{Conclusion}
\label{Sec: conclusion}
We have developed an approach to accelerate PDHG and ADMM in this paper. Our approach uses effective preconditioners to significantly reduce the number of iterations. In general, most effective preconditioners are non-diagonal and cause very difficult subproblems in PDHG and ADMM, so previous arts are restrictive with less effective diagonal preconditioners. However, we deal with those difficult subproblems by ``solving'' them highly inexactly, running just very few epochs of proximal BCD iterations. In all of our numerical tests, our algorithm needs relatively few outer iterations (due to effective preconditioners) and has the shortest total running time, achieving 7--95 times speedup over the next best algorithm.

Theoretically, we show a fixed number of inner iterations suffice for global convergence though a new relative error condition. The number depends on various factors but is easy to choose in all of our numerical results.


There are still open questions left for us to address in the future: (a) Depending on problem structures, there are  choices of preconditioners that are better than $M_1=\frac{1}{\tau}I_n, M_2=\tau AA^T$ (the ones that lead to ADMM if the subproblems are solved exactly). For example, in CT reconstruction, our choices of $M_1$ and $M_2$ have much faster overall convergence. (b) Is it possible to show Algorithm \ref{alg: Inexact PrePDHG} converges even with $S$ chosen as the iterator of faster accelerated solvers like
APCG \cite{lin2014accelerated}, NU\_ACDM \cite{allen2016even}, and A2BCD \cite{hannah2018texttt}?
(c) In general, how to accelerate a broader class of algorithms by integrating effective preconditioning and cheap inner loops while still ensuring global convergence? 
\appendix
\section{ADMM as a special case of PrePDHG}
\label{App: to ADMM}
\\

In this section we show that if we choose $M_1=\frac{1}{\tau}$ and $M_2=\tau AA^T$ in PrePDHG \eqref{equ: PrePDHG}, then it is equivalent to ADMM on the primal problem \eqref{equ: PDHG problem}.

By Theorem 1 of \cite{yan2016self}, we know that ADMM is primal-dual equivalent, in the sense that one can recover primal iterates from dual iterates and vice versa. Therefore, it suffices to show that $M_1=\frac{1}{\tau}$ and $M_2=\tau AA^T$ in PrePDHG \eqref{equ: PrePDHG} on the primal problem is equivalent to ADMM on the dual problem \eqref{equ: PDHG dual problem}.

In Theorem \ref{thm: Inexact PLADMM} we have shown that, under an appropriate change of variables, PrePDHG on the primal is equivalent to applying \eqref{equ: PLADMM} to the dual. As a result, we just need to demonstrate that the latter is exactly ADMM on the dual when
$M_1=\frac{1}{\tau}I_{n\times n}$ and $M_2=\tau AA^T$.

For the $z$-update in \eqref{equ: PLADMM}, we have
\begin{align}
    z^{k+1}
    &=\argmin_{z\in\Rm}\{g^*(z)-\tau\langle z-z^k, A(-A^Tz^k-y^k+u^k)\rangle+\frac{\tau}{2}\|z-z^k\|^2_{AA^T}\}\nonumber\\
    &=\argmin_{z\in\Rm}\{g^*(z)-\tau\langle z-z^k, A(-y^k+u^k)\rangle+\frac{\tau}{2}\|z\|^2_{AA^T}\}\nonumber\\
    &=\argmin_{z\in\Rm}\{g^*(z)+\tau\langle z, A(y^k-u^k)\rangle+\frac{\tau}{2}\|A^Tz\|^2\}\nonumber\\
    &=\argmin_{z\in\Rm}\{g^*(z)+\tau\langle A^Tz, -u^k\rangle+\frac{\tau}{2}\|A^Tz+y^k\|^2\}\nonumber\\
    &=\argmin_{z\in\Rm}\{g^*(z)+\tau\langle -A^Tz-y^k, u^k\rangle+\frac{\tau}{2}\|A^Tz+y^k\|^2\}.\label{equ: ADMM 1}
\end{align}
and for the $y$-update we have
\begin{align}
    y^{k+1}&=\prox^{M_1^{-1}}_{f^*}(u^{k}-A^Tz^{k+1})\nonumber\\
    &=\argmin_{y\in\Rn}\{f^*(y)+\frac{\tau}{2}\|y-u^k+A^Tz^{k+1}\|^2\}\nonumber\\
    &=\argmin_{y\in\Rn}\{f^*(y)+\tau\langle -A^Tz^{k+1}-y, u^k\rangle+\frac{\tau}{2}\|A^Tz^{k+1}+y\|^2\}.\label{equ: ADMM 2}
\end{align}
Define $v^k=\tau u^k$, \eqref{equ: ADMM 1}, \eqref{equ: ADMM 2}, and the $u-$update in \eqref{equ: PLADMM} become
\begin{align*}
    z^{k+1}&=\argmin_{z\in\Rm}\{g^*(z)+\langle -A^Tz-y^k, v^k\rangle+\frac{\tau}{2}\|A^Tz+y^k\|^2\},\\
    y^{k+1}&=\argmin_{y\in\Rn}\{f^*(y)+\langle -A^Tz^{k+1}-y, v^k\rangle+\frac{\tau}{2}\|A^Tz^{k+1}+y\|^2\},\\
    v^{k+1}&=v^{k}-\tau(A^Tz^{k+1}+y^{k+1}),
\end{align*}
which are ADMM iterations on the dual problem \eqref{equ: PDHG dual problem}.

\section{Proof of Theorem \ref{thm: BCD finite inner loops}: bounded relative error when $S$ is the iterator of cyclic proximal BCD}
\label{App: BCD proof}

The $z$-subproblem in \eqref{equ: PrePDHG} has the form
\[
\min_{z\in\Rm} h_1(z)+h_2(z),
\]
where
    $h_1(z) =g^*(z)=\sum_{i=1}^{l}g^*_i(z_i),$
    and $h_2(z)=\frac{1}{2}\|z-z^k-M_2^{-1}A(2x^{k+1}-x^k)\|^2_{M_2}.$
And $z^{k+1}=z^{k+1}_p$ is given by
\begin{align*}
    z^{k+1}_0&=z^{k},\\
    z^{k+1}_{i+1}&=S(z^{k+1}_i, x^{k+1}, x^k), \quad i=0,1,...,p-1,
\end{align*}
Here, $S$ is the iterator of cyclic proximal BCD. Define
\begin{align*}
T(z)&=\prox_{\gamma g^*(z)}(z-\gamma \nabla h_2(z))),\\
B(z) &= \frac{1}{\gamma}(z-T(z)),
\end{align*}
and the $i$th coordinate operator of $B$:
\[
B_i(z)=(0,...,(B(z))_i,...,0).
\]
Then, we have
\[
z^{k+1}_{i+1}=S(z^{k+1}_i, x^{k+1}, x^k)=(I-\gamma B_l)(I-\gamma B_2)...(I-\gamma B_1)z^{k+1}_i.
\]
By \cite[Prop. 26.16(ii)]{bauschke2017convex}, we know that $T(z)$ is a contraction with coefficient  $\theta=\sqrt{1-\gamma(2\lambda_{\mathrm{min}}(M_2)-\gamma \lambda_{\mathrm{max}}^2(M_2))}$. 
We know that for $\forall z_1, z_2\in\Rm$ and $\mu=\frac{1-\theta}{\gamma}$,
\begin{align*}
    \langle B(z_1)-B(z_2), z_1-z_2\rangle&=\frac{1}{\gamma}\|z_1-z_2\|^2-\frac{1}{\gamma}\langle T(z_1)-T(z_2), z_1-z_2\rangle\\
    &\geq \mu\|z_1-z_2\|^2,
\end{align*}

Let $z^{k+1}_{\star}=\argmin_{z\in \Rm}\{h_1(z)+h_2(z)\}$. For \cite[Thm 3.5]{chow2017cyclic}, we have
\begin{align}
\label{equ: BCD contraction}
\|z^{k+1}_i-z^{k+1}_{\star}\|\leq \rho^i \|z^{k+1}_0-z^{k+1}_{\star}\|, \quad \forall i=1,2,...,p.
\end{align}
where $\rho=1-\frac{\gamma \mu^2}{2}$.

Let $y_j=(I-\gamma B_j)...(I-\gamma B_1)z^{k+1}_{p-1}$ for $j=1,...,l$ and $y_0=z^{k+1}_{p-1}$. Note that $(z^{k+1}_{p})_j=(y_j)_j$ for $j=1,2,...,l$, and the blocks of $y_j$ satisfies
\begin{align*}
    (y_{j})_t=
    \begin{cases}
    \Big(\prox_{\gamma g^*}\big(y_{j-1}-\gamma \nabla h_2(y_{j-1})\big)\Big)_t, \quad \text{if}\,\,t=j \\
    (y_{j-1})_t, \quad\quad\quad\quad\quad\quad\quad\quad\quad\quad\quad\,\,\quad\, \text{otherwise}.
    \end{cases}
\end{align*}
On the other hand, we have
\[
\prox_{\gamma g^*}\big(y_{j-1}-\gamma \nabla h_2(y_{j-1})\big)=\argmin_{y\in\Rm}\{g^*(y)+\frac{1}{2\gamma}\|y-y_{j-1}+\gamma \nabla h_2(y_{j-1})\|^2\}.
\]
Since $g^*$ and $\|\cdot\|^2$ are separable, we obtain
\[
\vz \in \partial g^*_j((y_j)_j)+\frac{1}{\gamma}\Big((y_j)_j-(y_{j-1})_j+\gamma \big(\nabla h_2(y_{j-1})\big)_j\Big), \quad \forall j=1,2,...,l,
\]
or equivalently,
\[
\vz \in \partial g^*_j((z^{k+1}_p)_j)+\frac{1}{\gamma}\Big((z^{k+1}_p)_j-(z^{k+1}_{p-1})_j+\gamma \big(\nabla h_2(y_{j-1})\big)_j\Big), \quad \forall j=1,2,...,l.
\]
Therefore,
\[
\vz \in \partial g^*(z^{k+1}_p)+\frac{1}{\gamma}\Big(z^{k+1}_p-z^{k+1}_{p-1}+\gamma \xi_p\Big), \quad \forall j=1,2,...,l,
\]
where $(\xi_p)_j=\big(\nabla h_2(y_{j-1})\big)_j$ for $j=1,2,...,l$. Comparing this with \eqref{equ: varepsilon^{k+1}}, we obtain
\[
\varepsilon^{k+1}=\xi_p-\nabla h_2(z^{k+1}_p)+\frac{1}{\gamma}(z^{k+1}_p-z^{k+1}_{p-1}).
\]
Notice that the first $j-1$ blocks of $y_{j-1}$ are the same with those of $y_l=z^{k+1}_{p}$, and the rest of the blocks are the same with those of $y_0=z^{k+1}_{p-1}$, so we have
\begin{align}
    \|\varepsilon^{k+1}\|&\leq \sum_{j=1}^l \lambda_{\mathrm{max}}(M_2)\|y_{j-1}-z^{k+1}_p\|+\frac{1}{\gamma}\|z^{k+1}_p-z^{k+1}_{p-1}\|\nonumber\\
    &\leq l\lambda_{\mathrm{max}}(M_2)\|z^{k+1}_{p-1}-z^{k+1}_p\|+\frac{1}{\gamma}\|z^{k+1}_p-z^{k+1}_{p-1}\|\nonumber\\
    &\leq (l\lambda_{\mathrm{max}}(M_2)+\frac{1}{\gamma})(\|z^{k+1}_p-z^{k+1}_{\star}\|+\|z^{k+1}_{p-1}-z^{k+1}_{\star}\|)\nonumber
\end{align}
Combine this with \eqref{equ: BCD contraction}
\begin{align}
\label{equ: inequ 1}
    \|\varepsilon^{k+1}\|\leq (l\lambda_{\mathrm{max}}(M_2)+\frac{1}{\gamma})(\rho^p+\rho^{p-1})\|z^{k+1}_0-z^{k+1}_{\star}\|.
\end{align}
Combining
\begin{align*}
\|z^{k+1}-z^k\|&=\|z^{k+1}_p-z^{k+1}_0\|\\
&\geq \|z^{k+1}_{0}-z^{k+1}_{\star}\|-\|z^{k+1}_{p}-z^{k+1}_{\star}\|\\
&\geq (1-\rho^p)\|z^{k+1}_0-z^{k+1}_{\star}\|
\end{align*}
with \eqref{equ: inequ 1}, we obtain
\[
\|\varepsilon^{k+1}\|\leq \frac{(l\lambda_{\mathrm{max}}(M_2)+\frac{1}{\gamma})(\rho^p+\rho^{p-1})}{1-\rho^p}\|z^{k+1}-z^k\|.
\]

\section{Proof of Theorem \ref{thm: sequence convergence}: KŁ property gives global convergence}
\label{App: KL}

According to Theorem \ref{thm: Inexact PLADMM}, we just need to show that $\{M_1^{-1}u^k, z^k\}$ converges to a primal-dual solution pair of \eqref{equ: PDHG problem}.

By Theorem \ref{thm: subsequence convergence}, we can take $\{z^{k_s}, y^{k_s}, u^{k_s}\}\rightarrow (z^c, y^c, u^c)$ as $s\to\infty$. Note that
$L(z^{k_s}, y^{k_s}, u^{k_s})$ is monotonic nonincreasing and lower bounded due to Theorem \ref{thm: Lyapunov}, which implies the convergence of $L(z^{k_s}, y^{k_s}, u^{k_s})$. Since $L$ is lower semicontinuous, we have
\begin{align}
\label{equ: first inequ}
L(z^c, y^c, u^c)\leq \lim_{s\rightarrow \infty} L(z^{k_s}, y^{k_s}, u^{k_s}).
\end{align}
Since the only potentially discontinuous terms in $L$ is $g^*$, we have
\begin{align}
\label{equ: second inequ}
\lim_{s\rightarrow \infty} L(z^{k_s}, y^{k_s}, u^{k_s})-L(z^c, y^c, u^c)\leq \limsup_{s\rightarrow \infty} g^*(z^{k_s})-g^*(z^c).
\end{align}
By \eqref{equ: first optimality condition}, we know that \begin{align*}
    g^*(z^c)&\geq g^*(z^{k_s})\\
    &+\langle M_2(z^{k_s-1}-z^{k_s})+AM_1^{-1}(-A^Tz^{k_s-1}-y^{k_s-1}+u^{k_s-1})-\varepsilon^{k_s}, z^c-z^{k_s}\rangle,
\end{align*}
Then, yy Theorem \ref{thm: Lyapunov}, we further get $z^{k_s-1}-z^{k_s}\rightarrow \vz$. Since $z^{k_s}\rightarrow z^c$ and $\{z^{k}, y^{k}, u^{k}\}$ is bounded, we obtain
\[
\limsup_{s\rightarrow \infty} g^*(z^{k_s})-g^*(z^c)\leq 0.
\]
Combining this with \eqref{equ: first inequ} and \eqref{equ: second inequ}, we conclude that $\lim_{s\rightarrow \infty} L(z^{k_s}, y^{k_s}, u^{k_s})=L(z^c, y^c, u^c)$.

Since $g^*$ is a KŁ function, $L$ is also KŁ. Consequently, similar to Theorem 2.9 of \cite{attouch2013convergence}, we can claim the convergence of $\{z^k, y^k, u^k\}$ to $\{z^c, y^c, u^c\}$.

\section{Two-block ordering in Claim \ref{claim1} and four-block ordering in Claim \ref{claim2}}\label{App: Smart Ordering}

According to \eqref{equ: PrePDHG 1}, when $M_2=\tau AA^T$, the $z$-subproblem of Algorithm \ref{alg: Inexact PrePDHG} is
\begin{align}\label{equ: ordering 1}
{z}^{k+1}&=\argmin_{z\in\Rm}\{g^*(z)-\langle z-z^k, A(2x^{k+1}-x^k)\rangle+\frac{\tau}{2}\|A^T(z-z^k)\|^2_2\}.
\end{align}

Let us prove Claim \ref{claim1} first. In that claim, $A=\text{div}\in\R^{MN\times 2MN}$ and $z\in {\mathbb{R}}^{MN}$.
Following the definition of the sets $z_b$ and $z_r$,
we separate the $MN$ columns of $A^T=-D$ into two blocks $L_b$, $L_r$ by associating them with $z_b$ and $z_r$, respectively. Therefore, we have $A^Tz = L_bz_b+L_rz_r$ for any $z\in {\mathbb{R}}^{MN}$.

By the red-black ordering in Fig. \ref{2block}, different columns of $L_b$ are orthogonal one another, so ${L_b}^TL_b$ is diagonal. Similarly, ${L_r}^TL_r$ is also diagonal.

Let $b$ be the set of black nodes and $r$ the set of red nodes. We can rewrite \eqref{equ: ordering 1} as
\begin{align}
\label{equ: ordering 2}
{z}^{k+1}=\argmin_{z_b, z_r\in{\R^{MN/2}}}\{g_b^*(z_b)+g_r^*(z_r)+\langle z_b+z_r,c^k\rangle\\+\frac{\tau}{2}\|L_b(z_b-z_b^k)+L_r(z_r-z_r^k)\|^2_2\},\nonumber
\end{align}
where $g^*_b(z_b)=\sum_{(i,j)\in b} g^*_{i,j}(z_{i,j})$, $g^*_r(z_r)=\sum_{(i,j)\in r} g^*_{i,j}(z_{i,j})$, and $c^k=-A(2x^{k+1}-x^k)$.

Applying cyclic proximal BCD to black and red blocks alternatively yields
\begin{align}
    z_b^{k+\frac{t+1}{p}}&={\text{prox}}_{g_b^*(\cdot)+\langle\cdot,\tau {L_b}^TL_r(z_r^{k+\frac{t}{p}}-z_r^k)+c_b^k\rangle}^{\tau {L_b}^TL_b}(z_b^{k+\frac{t}{p}}),\\
    z_r^{k+\frac{t+1}{p}}&={\text{prox}}_{g_r^*(\cdot)+\langle\cdot,\tau {L_r}^TL_b(z_b^{k+\frac{t+1}{p}}-z_b^k)+c_r^k\rangle}^{\tau {L_r}^TL_r}(z_r^{k+\frac{t}{p}}).
\end{align} for $t=0,1,...,p-1$, where $p$ is the number of inner iterations in Algorithm \ref{alg: Inexact PrePDHG}.

These updates have closed-form solutions since $L_b^TL_b$ and $L_r^TL_r$ are diagonal, and all $prox_{\lambda g^*_{i,j}}$ are closed-form. Furthermore, the updates within each block can be done in parallel.

The proof of Claim \ref{claim2} is similar. When $A=D$ or $A=D_w$, we separate the columns of $A^T$ into four blocks $L_b$, $L_r$, $L_y$, $L_g$ by associating them with $z_b$, $z_r$, $z_y$ ,$z_g$, respectively. Therefore, we have $A^Tz = L_bz_b+L_rz_r+L_y z_y+L_g z_g$ for all $z\in {\mathbb{R}}^{2MN}$. Similarly, by the block design in Fig. \ref{4block}, cyclic proximal BCD iterations have closed-form solutions, and updates within each block can be executed in parallel.

\bibliographystyle{siamplain}
\bibliography{references}

\begin{thebibliography}{10}

\bibitem{allen2016even}
{\sc Z.~Allen-Zhu, Z.~Qu, P.~Richt{\'a}rik, and Y.~Yuan}, {\em Even faster
  accelerated coordinate descent using non-uniform sampling}, in International
  Conference on Machine Learning, 2016, pp.~1110--1119.

\bibitem{attouch2013convergence}
{\sc H.~Attouch, J.~Bolte, and B.~F. Svaiter}, {\em Convergence of descent
  methods for semi-algebraic and tame problems: proximal algorithms,
  forward-backward splitting, and regularized {G}auss-{S}eidel methods},
  Mathematical Programming, 137 (2013), pp.~91--129.

\bibitem{bauschke2017convex}
{\sc H.~H. Bauschke, P.~L. Combettes, et~al.}, {\em Convex Analysis and
  Monotone Operator Theory in Hilbert Spaces}, vol.~2011, Springer, 2017.

\bibitem{becker2011templates}
{\sc S.~R. Becker, E.~J. Cand{\`e}s, and M.~C. Grant}, {\em Templates for
  convex cone problems with applications to sparse signal recovery},
  Mathematical programming computation, 3 (2011), p.~165.

\bibitem{bredies2015preconditioned}
{\sc K.~Bredies and H.~Sun}, {\em Preconditioned {D}ouglas-{R}achford splitting
  methods for convex-concave saddle-point problems}, SIAM Journal on Numerical
  Analysis, 53 (2015), pp.~421--444.

\bibitem{bredies2017proximal}
{\sc K.~Bredies and H.~Sun}, {\em A proximal point analysis of the
  preconditioned alternating direction method of multipliers}, Journal of
  Optimization Theory and Applications, 173 (2017), pp.~878--907.

\bibitem{chambolle2011first}
{\sc A.~Chambolle and T.~Pock}, {\em A first-order primal-dual algorithm for
  convex problems with applications to imaging}, Journal of mathematical
  imaging and vision, 40 (2011), pp.~120--145.

\bibitem{chambolle2016ergodic}
{\sc A.~Chambolle and T.~Pock}, {\em On the ergodic convergence rates of a
  first-order primal--dual algorithm}, Mathematical Programming, 159 (2016),
  pp.~253--287.

\bibitem{chow2017cyclic}
{\sc Y.~T. Chow, T.~Wu, and W.~Yin}, {\em Cyclic coordinate-update algorithms
  for fixed-point problems: Analysis and applications}, SIAM Journal on
  Scientific Computing, 39 (2017), pp.~A1280--A1300.

\bibitem{combettes2013moreau}
{\sc P.~L. Combettes and N.~N. Reyes}, {\em Moreau’s decomposition in banach
  spaces}, Mathematical Programming, 139 (2013), pp.~103--114.

\bibitem{cvx}
{\sc I.~CVX~Research}, {\em {CVX}: Matlab software for disciplined convex
  programming, version 2.0}.
\newblock \url{http://cvxr.com/cvx}, Aug. 2012.

\bibitem{eckstein1992douglas}
{\sc J.~Eckstein and D.~P. Bertsekas}, {\em On the {D}ouglas-{R}achford
  splitting method and the proximal point algorithm for maximal monotone
  operators}, Mathematical Programming, 55 (1992), pp.~293--318.

\bibitem{eckstein2017approximate}
{\sc J.~Eckstein and W.~Yao}, {\em Approximate {ADMM} algorithms derived from
  lagrangian splitting}, Computational Optimization and Applications, 68
  (2017), pp.~363--405.

\bibitem{eckstein2017relative}
{\sc J.~Eckstein and W.~Yao}, {\em Relative-error approximate versions of
  {D}ouglas-{R}achford splitting and special cases of the {ADMM}}, Mathematical
  Programming,  (2017), pp.~1--28.

\bibitem{esser2010general}
{\sc E.~Esser, X.~Zhang, and T.~F. Chan}, {\em A general framework for a class
  of first order primal-dual algorithms for convex optimization in imaging
  science}, SIAM Journal on Imaging Sciences, 3 (2010), pp.~1015--1046.

\bibitem{feijer2010stability}
{\sc D.~Feijer and F.~Paganini}, {\em Stability of primal--dual gradient
  dynamics and applications to network optimization}, Automatica, 46 (2010),
  pp.~1974--1981.

\bibitem{gauss1903werke}
{\sc C.~F. Gauss}, {\em Werke (in German), 9, Göttingen: Köninglichen
  Gesellschaft der Wissenschaften}, 1903.

\bibitem{giselsson2014diagonal}
{\sc P.~Giselsson and S.~Boyd}, {\em Diagonal scaling in {D}ouglas-{R}achford
  splitting and {ADMM}}, in Decision and Control (CDC), 2014 IEEE 53rd Annual
  Conference on, IEEE, 2014, pp.~5033--5039.

\bibitem{giselsson2017linear}
{\sc P.~Giselsson and S.~Boyd}, {\em Linear convergence and metric selection
  for {D}ouglas-{R}achford splitting and {ADMM}}, IEEE Transactions on
  Automatic Control, 62 (2017), pp.~532--544.

\bibitem{gb08}
{\sc M.~Grant and S.~Boyd}, {\em Graph implementations for nonsmooth convex
  programs}, in Recent Advances in Learning and Control, V.~Blondel, S.~Boyd,
  and H.~Kimura, eds., Lecture Notes in Control and Information Sciences,
  Springer-Verlag Limited, 2008, pp.~95--110.
\newblock \url{http://stanford.edu/~boyd/graph_dcp.html}.

\bibitem{hannah2018texttt}
{\sc R.~Hannah, F.~Feng, and W.~Yin}, {\em {A}2{BCD}: An asynchronous
  accelerated block coordinate descent algorithm with optimal complexity},
  arXiv preprint arXiv:1803.05578,  (2018).

\bibitem{hansen2018air}
{\sc P.~C. Hansen and J.~S. J{\o}rgensen}, {\em Air tools ii: algebraic
  iterative reconstruction methods, improved implementation}, Numerical
  Algorithms, 79 (2018), pp.~107--137.

\bibitem{levina2001earth}
{\sc E.~Levina and P.~Bickel}, {\em The earth mover's distance is the mallows
  distance: Some insights from statistics}, in null, IEEE, 2001, p.~251.

\bibitem{li2013inexact}
{\sc M.~Li, L.-Z. Liao, and X.~Yuan}, {\em Inexact alternating direction
  methods of multipliers with logarithmic-quadratic proximal regularization},
  Journal of Optimization Theory and Applications, 159 (2013), pp.~412--436.

\bibitem{li2017parallel}
{\sc W.~Li, E.~K. Ryu, S.~Osher, W.~Yin, and W.~Gangbo}, {\em A parallel method
  for earth mover's distance}, UCLA Comput. Appl. Math. Pub.(CAM) Rep,  (2017),
  pp.~17--12.

\bibitem{lin2014accelerated}
{\sc Q.~Lin, Z.~Lu, and L.~Xiao}, {\em An accelerated proximal coordinate
  gradient method}, in Advances in Neural Information Processing Systems, 2014,
  pp.~3059--3067.

\bibitem{lin2015accelerated}
{\sc Q.~Lin, Z.~Lu, and L.~Xiao}, {\em An accelerated randomized proximal
  coordinate gradient method and its application to regularized empirical risk
  minimization}, SIAM Journal on Optimization, 25 (2015), pp.~2244--2273.

\bibitem{metivier2016measuring}
{\sc L.~M{\'e}tivier, R.~Brossier, Q.~M{\'e}rigot, E.~Oudet, and J.~Virieux},
  {\em Measuring the misfit between seismograms using an optimal transport
  distance: application to full waveform inversion}, Geophysical Supplements to
  the Monthly Notices of the Royal Astronomical Society, 205 (2016),
  pp.~345--377.

\bibitem{ng2011inexact}
{\sc M.~K. Ng, F.~Wang, and X.~Yuan}, {\em Inexact alternating direction
  methods for image recovery}, SIAM Journal on Scientific Computing, 33 (2011),
  pp.~1643--1668.

\bibitem{o2015adaptive}
{\sc B.~O’donoghue and E.~Candes}, {\em Adaptive restart for accelerated
  gradient schemes}, Foundations of computational mathematics, 15 (2015),
  pp.~715--732.

\bibitem{pele2009fast}
{\sc O.~Pele and M.~Werman}, {\em Fast and robust earth mover's distances.}, in
  ICCV, vol.~9, 2009, pp.~460--467.

\bibitem{pock2011diagonal}
{\sc T.~Pock and A.~Chambolle}, {\em Diagonal preconditioning for first order
  primal-dual algorithms in convex optimization}, in Computer Vision (ICCV),
  2011 IEEE International Conference on, IEEE, 2011, pp.~1762--1769.

\bibitem{pock2009algorithm}
{\sc T.~Pock, D.~Cremers, H.~Bischof, and A.~Chambolle}, {\em An algorithm for
  minimizing the {M}umford-{S}hah functional}, in Computer Vision, 2009 IEEE
  12th International Conference on, IEEE, 2009, pp.~1133--1140.

\bibitem{rasch2018inexact}
{\sc J.~Rasch and A.~Chambolle}, {\em Inexact first-order primal-dual
  algorithms}, arXiv preprint arXiv:1803.10576,  (2018).

\bibitem{richardson1911ix}
{\sc L.~F. Richardson}, {\em Ix. the approximate arithmetical solution by
  finite differences of physical problems involving differential equations,
  with an application to the stresses in a masonry dam}, Phil. Trans. R. Soc.
  Lond. A, 210 (1911), pp.~307--357.

\bibitem{rockafellar2009variational}
{\sc R.~T. Rockafellar and R.~J.-B. Wets}, {\em Variational Analysis},
  vol.~317, Springer Science \& Business Media, 2009.

\bibitem{saad2003iterative}
{\sc Y.~Saad}, {\em Iterative Methods for Sparse Linear Systems}, vol.~82,
  SIAM, 2003.

\bibitem{sidky2012convex}
{\sc E.~Y. Sidky, J.~H. J{\o}rgensen, and X.~Pan}, {\em Convex optimization
  problem prototyping for image reconstruction in computed tomography with the
  chambolle--pock algorithm}, Physics in Medicine \& Biology, 57 (2012),
  p.~3065.

\bibitem{valkonen2014primal}
{\sc T.~Valkonen}, {\em A primal--dual hybrid gradient method for nonlinear
  operators with applications to mri}, Inverse Problems, 30 (2014), p.~055012.

\bibitem{wang2015global}
{\sc Y.~Wang, W.~Yin, and J.~Zeng}, {\em Global convergence of {ADMM} in
  nonconvex nonsmooth optimization}, arXiv preprint arXiv:1511.06324,  (2015).

\bibitem{yan2016self}
{\sc M.~Yan and W.~Yin}, {\em Self equivalence of the alternating direction
  method of multipliers}, in Splitting Methods in Communication, Imaging,
  Science, and Engineering, Springer, 2016, pp.~165--194.

\bibitem{zhang2011unified}
{\sc X.~Zhang, M.~Burger, and S.~Osher}, {\em A unified primal-dual algorithm
  framework based on bregman iteration}, Journal of Scientific Computing, 46
  (2011), pp.~20--46.

\bibitem{zhu2008efficient}
{\sc M.~Zhu and T.~Chan}, {\em An efficient primal-dual hybrid gradient
  algorithm for total variation image restoration}, UCLA CAM Report, 34 (2008).

\end{thebibliography}
\end{document}